\documentclass[12pt,reqno]{article}

\usepackage{xcolor}
\usepackage{amssymb}
\usepackage{amsmath}
\usepackage{amsthm}
\usepackage{amsfonts}
\usepackage{amscd}
\usepackage{graphicx}
\usepackage{mathrsfs}
\usepackage{mathtools}
\usepackage{multicol} 

\definecolor{webgreen}{rgb}{0,.5,0}
\definecolor{webbrown}{rgb}{.6,0,0}

\usepackage[colorlinks=true,
linkcolor=webgreen,
filecolor=webbrown,
citecolor=webgreen]{hyperref}

\usepackage{fullpage}
\usepackage{float}
\usepackage{graphics}
\usepackage{latexsym}
\usepackage{epsf}
\usepackage{algorithm}
\usepackage{algorithmic}
\usepackage{enumerate,tikz,listings,array,longtable}
\usetikzlibrary{shapes,arrows,chains}
\usepackage{pgf}

\newcommand{\seqnum}[1]{\href{https://oeis.org/#1}{\rm \underline{#1}}}

\begin{document}

\theoremstyle{plain}
\newtheorem{theorem}{Theorem}
\newtheorem{corollary}[theorem]{Corollary}
\newtheorem{lemma}[theorem]{Lemma}
\newtheorem{proposition}[theorem]{Proposition}

\theoremstyle{definition}
\newtheorem{definition}[theorem]{Definition}
\newtheorem{example}[theorem]{Example}
\newtheorem{conjecture}[theorem]{Conjecture}

\theoremstyle{remark}
\newtheorem{remark}[theorem]{Remark}

\begin{center}
\vskip 1cm{\LARGE\bf On Flattened Parking Functions}
\vskip 1cm
\large
Jennifer Elder\\
Department of Computer Science, Mathematics and Physics\\ Missouri Western State University\\ 
4525 Downs Drive\\
St.\ Joseph, MO 64507\\
USA\\
\href{mailto:flattenedparkingfunctions@outlook.com}{\tt flattenedparkingfunctions@outlook.com}\\
\ \\
Pamela E. Harris\\
Department of Mathematical Sciences\\ University of Wisconsin-Milwaukee\\
3200 N.\ Cramer Street\\
Milwaukee, WI 53211\\
USA\\
\href{mailto: peharris@uwm.edu}{\tt peharris@uwm.edu}
\\
\ \\
Zoe Markman\\
Department of Mathematics and Statistics\\ Swarthmore College\\
500 College Avenue\\
Swarthmore, PA 19081\\
USA\\
\href{mailto: zmarkma1@swarthmore.edu}{\tt zmarkma1@swarthmore.edu}\\
\ \\
Izah Tahir\\
School of Mathematics\\Georgia Institute of Technology\\
686 Cherry Street\\
Atlanta, GA 30332\\
USA\\
\href{mailto: itahir3@gatech.edu}{\tt itahir3@gatech.edu}\\
\ \\
Amanda Verga\\
Department of Mathematics\\ Trinity College\\
300 Summit Street\\
Hartford, CT 06106\\
USA\\
\href{mailto:amanda.verga@trincoll.edu}{\tt amanda.verga@trincoll.edu}
\end{center}

\vskip .2in

\begin{abstract}
A permutation of length $n$ is called a flattened partition if the leading terms of maximal chains of ascents (called runs) are in increasing order.
We analogously define 
flattened parking functions: a subset of parking functions for which the leading terms of maximal chains of weak ascents (also called runs) are in weakly increasing order.
For $n\leq 8$, where there are at most four runs, we give data for the number of flattened parking functions, and it remains an open problem to give formulas for their enumeration in general.
We then specialize to a subset of flattened parking functions that we call $\mathcal{S}$-insertion flattened parking functions. 
These are obtained by inserting all numbers of a multiset $ \mathcal{S}$ whose elements are in $[n]=\{1,2,\ldots,n\}$, 
into a permutation of $[n]$ and checking that the result is flattened.
We provide bijections between  $\mathcal{S}$-insertion flattened parking functions and $\mathcal{S}'$-insertion flattened  parking functions, where $\mathcal{S}$ and $\mathcal{S}'$ have certain relations.
We then further specialize to the case $\mathcal{S}=\textbf{1}_r$, the multiset with $r$ ones, and we establish a bijection between  $\textbf{1}_r$-insertion flattened parking functions and set partitions of $[n+r]$ with the first $r$ integers in different subsets. 
\end{abstract}

\section{Introduction}
Let $n\in\mathbb{N}\coloneqq\{1,2,3,\ldots\}$ and let $\mathfrak{S}_{n}$ denote the set of permutations of $[n]\coloneqq\{1,2,\ldots,n\}$. Recall that a permutation $\sigma = \sigma_1 \sigma_2\cdots \sigma_n\in\mathfrak{S}_n$ has a \emph{descent} (an \emph{ascent}) at index $i$ if $\sigma_i>\sigma_{i+1}$ (if $\sigma_i<\sigma_{i+1}$). Ascents and descents are often referred to as a \emph{permutation statistic}.
The study of permutation statistics is a robust area of research in combinatorics and has gathered much interest as its connections to other areas of math are plentiful \cite{ Harris1, JEetal, TMMSSW, JS, YZ}.

Another statistic of interest is the number of runs of a permutation. 
We recall that a \textit{run of length $p+1$} in $ \sigma\in\mathfrak{S}_n $ is a subword $ \sigma_{i} \sigma_{i+1} \cdots \sigma_{i+p-1} \sigma_{i+p} $ where $ i, i+1, \ldots, i + p-1 $ are consecutive ascents, and $ \sigma_{i+p} > \sigma_{i+p+1} $, or $i+p=n$. Given $\sigma\in\mathfrak{S}_n$, we let $\text{run}(\sigma)$ denote the number of runs.
Zhuang \cite{YZ} counted permutations by runs and their work provided a generalization of a theorem of Gessel, which gave a reciprocity formula for noncommutative symmetric functions. Gessel \cite{IG} also enumerated permutations having certain restrictions on their run lengths. 
Related work considered permutations whose runs satisfy the condition that the leading terms form an increasing sequence; we define this formally next.

A permutation $ \sigma $ is said to be a \textit{flattened partition} if it consists of runs arranged from left to right such that their leading values are in increasing order. For example, the permutation $\sigma = 1423$ is flattened, as its runs are $14$ and $23$, and $1\leq 2$. However, the permutation $\tau = 4321$ is not flattened, as each element is in its own run, and these elements form a decreasing sequence.
The research teams of Nabawanda, Rakotondrajao, Bamunoba \cite{NaRaBa} and Beyene, Mantaci \cite{FBRM} studied the distribution of runs for flattened partitions.
Nabawanda et al.\ establish a new combinatorial interpretation for the number of flattened partitions with $k$ runs along with a bijection between this combinatorial family of objects and set partitions of $[n-1] $.

In the current work, we extend the definition of flattened partitions to parking functions by allowing ascents to be  \emph{weak} ascents, i.e., a word $a_1a_2\cdots a_n$ has a \emph{weak ascent} at index $i$ if $a_i\leq a_{i+1}$.
We begin by recalling that parking functions are tuples in $[n]^n$ describing the parking preferences of $n$ cars (denoted $c_1,c_2,\ldots,c_n$) entering a one-way street (in sequential order) with a preference for a parking spot among $n$ spots on the street. 
If $\alpha=(a_1,a_2,\ldots,a_n)\in[n]^n$ is such that $a_i$ denotes the parking preference of car $c_i$, then $c_i$ drives to $a_i$ and parks there if the spot is not previously occupied by an earlier car $c_j$, with $1\leq j<i$. 
If such a car has already parked in spot $a_i$, then $c_i$ drives down the street parking in the first available spot. 
Given the parking preferences in $\alpha$, if the $n$ cars are able to park in the $n$ spots on the street, then we say that $\alpha$ is a parking function.
We let $\textnormal{PF}_n$ denote the set of all parking functions of length $n$ and recall from Konheim and Weiss \cite{KW} that $|\textnormal{PF}_n|=(n+1)^{n-1}$. 
Of course, if every entry of $\alpha$ is distinct (i.e.,~ $\alpha$ is a permutation of $n$), then all of the cars can park. Thus the set of permutations on $n$ letters is a subset of $\textnormal{PF}_n$. 

There is much known about parking functions, including their enumeration, their characterization in terms of inequalities, and statistics such as ascents and descents from Schumacher \cite{SCH}, not to mention their numerous generalizations, such as the work done by Yan \cite{Yan}. 
However, there is little to nothing known about the set of flattened parking functions of length $n$:
\[\textnormal{flat}(\textnormal{PF}_n)\coloneqq\{\alpha\in\textnormal{PF}_n:
\mbox{the leading values in the runs of}\, \alpha\, \mbox{are in weakly increasing order}
\},\]
or of the set of all flattened parking functions of length $n$ with $k$ runs:
\[\textnormal{flat} _k(\textnormal{PF}_n)\coloneqq\{\alpha\in \textnormal{flat}(\textnormal{PF}_n):\alpha\,\mbox{has $k$ runs}
\}. \]
Given a parking function $\alpha$, we say $\alpha$ is \emph{flattened} whenever $\alpha\in\textnormal{flat} (\textnormal{PF}_n)$ and \emph{$k$-flattened} whenever $\alpha\in\textnormal{flat} _k(\textnormal{PF}_n)$. 
Moreover, we abuse notation and let $\text{run}(\alpha)$ denote the number of runs of $\alpha$.
Thus, the goal of this paper is to extend the study of run distributions of flattened partitions to parking functions.

Our preliminary results establish that the set $\textnormal{flat} _1(\textnormal{PF}_n)$ is enumerated by the Catalan numbers, \seqnum{A000108}, (Theorem \ref{thm:CatalanBijection}), 
and we show that the maximum number of runs in a flattened parking function of length $n$ is at most $\lceil\frac{n}{2}\rceil$, see (Theorem \ref{thm:MaxNumberofRuns}).
In Table \ref{flattenedpf}, we provide  data on the number of flattened parking functions of length $ n $, which we then aggregate by $k$, the number of runs.
We note that the sequence $ (\textnormal{flat} (\textnormal{PF}_n))_{n\geq1} $ does not appear as a known sequence in the Online Encyclopedia of Integer Sequences (OEIS) \cite{OEIS}. 
In fact, it remains an open problem to provide a recursive formula for this sequence. Given this challenge, we restrict our study by considering a specific subset of parking functions that are created from permutations and (through an insertion process) made into parking functions. 
\begin{center}
\begin{table}[H]
\centering
\begin{tabular}{|l|l||l|l|l|l|} 
\hline
$n$& $|\textnormal{flat}(\textnormal{PF}_n)|$ & $|\textnormal{flat} _1(\textnormal{PF}_n)|$ &$|\textnormal{flat} _2(\textnormal{PF}_n)|$ & $|\textnormal{flat} _3(\textnormal{PF}_n)|$ & $|\textnormal{flat} _4(\textnormal{PF}_n)|$ \\ 
\hline
1 & 1 & 1 & 0 & 0 & 0 \\ 
2 & 2 & 2 & 0 & 0 & 0 \\ 
3 & 8 & 5 & 3 & 0 & 0 \\ 
4 & 46 & 14 & 32 & 0 & 0 \\ 
5 & 336 & 42 & 245 & 49 & 0 \\ 
6 & 2937 & 132 & 1656 & 1149 & 0 \\ 
7 & 29629 & 429 & 10563 & 17008 & 1629 \\
8 & 336732 & 1430 & 65472 & 204815 & 65015 \\ 
\hline
\end{tabular}
\caption{Total flattened parking functions of length $ n $, aggregated by $k$, the number of runs.}\label{flattenedpf}
\end{table}
\end{center}

Let $\mathcal{S}$ be a multiset whose elements are in $[n+1]$ and let $|\mathcal{S}|$ denote the cardinality of $\mathcal{S}$ with multiplicity. 
Given a permutation 
$\pi\in\mathfrak{S}_n$, let 
$\mathcal{I}(\mathcal{S},\pi)$ be the set of words of length $n+|S|$ constructed by inserting the elements of $\mathcal{S}$ into the permutation $\pi$ in all possible ways. 
For example, if $\mathcal{S}=\{1,1\}$ and $\pi=12$, then  $\mathcal{I}(\mathcal{S},\pi)=\{1112,1121,1211\}$.
Then, we define 
the set of \textit{$ \mathcal{S} $-insertion parking functions} as the set
\[ \mathcal{PF}_n(\mathcal{S}) \coloneqq \bigcup_{\pi\in\mathfrak{S}_n}\mathcal{I}(\mathcal{S},\pi) .\]
We show in (Lemma \ref{thm:insertionAreParkingFxn}) that for all multisets 
$\mathcal{S}$ with elements in $[n+1]$, the set $\mathcal{PF}_n(\mathcal{S}) $ 
is a subset of $\textnormal{PF}_{n+|\mathcal{S}|}$.
A parking function $\alpha\in\mathcal{PF}_n(\mathcal{S}) $ is said to be a \textit{ $\mathcal{S} $-insertion flattened parking function} if it satisfies the properties of a flattened parking function; we let \[\textnormal{flat} (\mathcal{PF}_n(\mathcal{S}))\coloneqq\{\alpha\in\mathcal{PF}_n(\mathcal{S}):\alpha\mbox{ is flattened}\} \]
denote the set of $\mathcal{S}$-insertion flattened parking functions and we let
\[\textnormal{flat} _k(\mathcal{PF}_n(\mathcal{S}))\coloneqq\{\alpha\in\mathcal{PF}_n(\mathcal{S}):\alpha\mbox{ is flattened with $k$ runs}\}\]
denote the set of $\mathcal{S}$-insertion flattened parking functions with exactly $k$ runs.
For ease of notation, henceforth we let $|\textnormal{flat} (\mathcal{PF}_n(\mathcal{S}))|=f(\mathcal{S}; n)$ and $|\textnormal{flat} _k(\mathcal{PF}_n(\mathcal{S}))|=f(\mathcal{S}; n,k)$.

Given these definitions, in Section \ref{section:general_S_insertion}, we  provide bijections between sets $\textnormal{flat} (\mathcal{PF}_n (\mathcal{S}))$ and $\textnormal{flat}( \mathcal{PF}_n (\mathcal{S}'))$ with $\mathcal{S}$ and $\mathcal{S}'$ having certain relations.
In Section~\ref{Sec:OnlyOneOne}, we focus on flattened 
$\mathbf{1}_r$-insertion flattened parking functions, where $\mathbf{1}_{r}\coloneqq\{1,\ldots,1\}$ with $|\mathbf{1}_r|=r$.
When $r=1$, our results generalize many of the results provided by \cite{NaRaBa} from flattened partitions to  $\{1\}$-insertion flattened parking functions.
We then extend our study to $r\geq 2$, in Theorem~\ref{rkBellNumbers} we give a bijection between the elements of $\textnormal{flat}_{k+1}(\mathcal{PF}_{n+1}(\mathbf{1}_r))$ and 
set of set partitions of $ [n+r] $ satisfying the following conditions:
\begin{enumerate}
\item the first $ r $ numbers are in different subsets, and 
\item there are exactly $k$ blocks with at least two elements.
\end{enumerate}
We define the $(r,k)$-Bell numbers, denoted $B_{k}(n,r)$, as the cardinality of the set partitions described above, which generalize the sequence found in \seqnum{A124234} for $r>1$.
We remark that in our work we establish that the statistic $k$, which accounts for $k+1$ runs in the  $\mathbf{1}_r$-insertion flattened parking functions and the $k$ blocks of size at least 2 in the set partitions is preserved by our bijections. 
While Nabawanda et al.\ noted that the bijection was statistic preserving, they did not note that $f_{n,k}$ is enumerated by \seqnum{A124234}.

With those results at hand, in Section \ref{sec:GenFunc},
we attempted to generalize the work of Nabawanda et al.\ and study the exponential generating function of the numbers $|\textnormal{flat} _k(\mathcal{PF}_n(\mathbf{1}_r))|=f(\mathbf{1}_{r};n,k)$:
\[F(x,y,z) = \sum_{n\geq 0}\;\sum_{k\geq 0} \;\sum_{r\geq 1} f(\mathbf{1}_{r};n,k)x^{k}\dfrac{y^n}{n!}\frac{z^r}{r!} = \sum_{n\geq 0} \;\sum_{r\geq 1}\; \sum_{\alpha \in \textnormal{flat} (\mathcal{PF}_n(\mathbf{1}_{r}))} x^{\textnormal{run} (\alpha)} \dfrac{y^n}{n!}\frac{z^r}{r!}.\] 
However, we identify an error in the computation for the closed differential form provided in \cite[p.\ 8]{NaRaBa}. We provide the details of this error, which we were unable to resolve. We leave it as an open problem to find or prove that no closed differential form exists, as well as then consider the corresponding generating function for our generalizations.

Section \ref{sec:flattened_s_in_dif_runs} provides our final set of results, where we generalize the results in Nabawanda et al.\ \cite{FBRM,NaRaBa} to  $\mathbf{1}_{r}$-insertion flattened parking functions whose first $s$ integers belong to different runs. 
In particular, we provide a generalization
which involves compositions of $r$ delineating the position of ones among the $k$ runs.
To conclude, we provide a variety of directions for future work.

\section{Flattened parking functions}

Throughout this paper, we work with both permutations and parking functions. For this reason, we abuse notation slightly in the following way: for $\alpha = (a_1, a_2,\ldots, a_n) \in \textnormal{PF}_n$, we write $\alpha = a_1a_2\cdots a_n$. This way of writing the parking function as a word allows us to more easily go back and forth between parking functions and permutations, whose one line notation is written in this form.

As with permutations, the set of all parking functions of length $n$ grows rapidly. This is why we study new classes of parking functions with particular properties. 

Recall we generalized the notion of a run from permutations to parking functions by allowing weak ascents. This allows us to define the following class of parking functions.

\begin{definition}
A parking function $ \pi $ is said to be a \textit{flattened parking function} if it consists of runs arranged from left to right such that their leading values are in weakly increasing order. 
\end{definition}
For example, note that the parking function
$112233456$ has a single run and is flattened and the parking function $14224222$ is also flattened and has three runs. However, the parking function $7654321$ is not flattened. 
We denote the set of all flattened parking functions of length $ n $ as $\textnormal{flat}(\mathcal{PF}_n)$, and the set of flattened parking functions of length $ n $ with $ k $ runs as $ \textnormal{flat}_k(\mathcal{PF}_n)$. 
Table \ref{flattenedpf} provides data on the cardinalities of these sets for $1\leq k\leq 4$ and $1\leq n\leq 8$. Our first result connects these sets with the Catalan numbers $C_n=\frac{1}{n+1}\binom{2n}{n}$. 

\begin{theorem} \label{thm:CatalanBijection}
If $n\geq 1$, then $|\textnormal{flat} _1(\textnormal{PF}_n)|=C_n$.
\end{theorem}

\begin{proof}
Flattened parking functions with one run are the same set as all weakly increasing parking functions, as a descent would mark the start of a new run. Armstrong, Loehr, and Warrington \cite{DAet} proved that weakly increasing parking functions are in bijection with the Catalan numbers. 
\end{proof}

\begin{theorem}\label{thm:MaxNumberofRuns}
The maximal number of runs for a flattened parking function of length $ n $ is~$ \left\lceil\frac{n}{2}\right\rceil $.
\end{theorem}

\begin{proof}
Note that there cannot be a run of length one in the middle of a flattened parking function; if there is a run of length one aside from the last run, the leading value must be greater than the subsequent leading value. 
Hence the leading values would not fulfill the conditions for being flattened. So, every run besides the last run must have length at least two. 

We now consider the parity of $n$. 
First, if $ n $ is odd, then there can only be one run of length one, and all other runs must be of length at least two. So, the maximal number of runs is given by
\begin{equation*}
\frac{n-1}{2} + 1 = \left\lceil\frac{n}{2}\right\rceil.
\end{equation*}

In the case where $ n $ is even, there can only be one run of length one. However, if there is one run of odd length, there must be a second run of odd length: in which case we have either one run of length one and one run of length three or more, or no runs of length one. Observe that having one run of length one and one run of length three does not yield a greater number of runs than the flattened parking function with all runs of length two. 
Thus, we can split $ n $ into runs of length two to get a maximal number of runs which is given by $ \frac{n}{2} $.
This is equivalent to~$ \left\lceil\frac{n}{2}\right\rceil $.
\end{proof}

With these initial results at hand we are now able to begin our study of $\mathcal{S}$-insertion flattened parking functions.

\section{\texorpdfstring{$\mathcal{S}$}{Lg}-insertion flattened parking functions}\label{section:general_S_insertion}

Similarly to how we restricted our focus from parking functions to flattened parking functions, we restrict further to a smaller subclass. First, we define a process by which we can construct parking functions from permutations.

Let $\mathfrak{S}_n$ denote the set of permutations on $[n]$ and let
$\mathcal{S}$ be a multiset whose elements are in $[n+1]$ and let $|\mathcal{S}|$ denote the cardinality of $\mathcal{S}$ with multiplicity.

\begin{definition}\label{def:S_Insertion_Parking_Fun}
For a multiset $\mathcal{S}$ whose elements are in $[n+1]$, let $\mathcal{I}(\mathcal{S}, \pi)$ be the set of words of length $n+|S|$ constructed by inserting the elements of $\mathcal{S}$ into the permutation $\pi\in \mathfrak{S}_n$ in all possible ways. Denote the set of \textit{$ \mathcal{S} $-insertion parking functions} as the set
$ \mathcal{PF}_n(\mathcal{S}) \coloneqq \bigcup_{\pi\in\mathfrak{S}_n}\mathcal{I}(\mathcal{S}, \pi) $. 
\end{definition}

\begin{remark}
One might believe that the $\mathcal{S}$-insertion procedure would yield all of the parking functions of length $n+|\mathcal{S}|$. However, this is not true. For example, the parking function $111111119$ of length $9$ cannot be created from a permutation in $\mathfrak{S}_8$, as the smallest possible permutation in this list is the identity element on one letter, namely $\pi=1\in\mathfrak{S}_1$. Moreover, our insertion process would only allow $\mathcal{S}$ to contain the elements one and two, thereby never being able to insert the number~$9$. 
\end{remark}

We now show that all of the words created through the $\mathcal{S}$-insertion process in Definition~\ref{def:S_Insertion_Parking_Fun} produce parking functions.

\begin{theorem} The set $ \mathcal{PF}_{n}(\mathcal{S}) \coloneqq \bigcup_{\pi\in\mathfrak{S}_{n}}\mathcal{I}(\mathcal{S},\pi) $ is a subset of $\mathcal{PF}_{n+|\mathcal{S}|}$.
\end{theorem} \label{thm:insertionAreParkingFxn}

\begin{proof}
Recall the well-known fact that a list of numbers $x_1x_2\cdots x_n$ with entries in $[n]$ is a parking function if and only if its weakly increasing rearrangement $x_1'x_2'\cdots x_n'$ satisfies $x_i'\leq i$ for all $i\in[n]$.
Now notice that if $a = a_1a_2\cdots a_n$ is a permutation in $\mathfrak{S}_n$, then in the extreme case where we insert $\mathcal{S} = \{n+1, n+1, \ldots, n+1\}$, with $|\mathcal{S}|=m$, into $a$ yields a word whose weakly increasing rearrangement $a'$ satisfies $ a_{i}' \leq i $ for all $i\in[n+m]$. 
Thus, inserting $\mathcal{S}$ into $a$ would result in a parking function of length $n+m=n+|\mathcal{S}|$, as desired. 
We conclude by noting that all other multisets $\mathcal{S}$ would result in the same inequalities being satisfied.
\end{proof}
We now extend the definition of ``flattened'' to the set $\mathcal{PF}_n(\mathcal{S})$.

\begin{definition}\label{def:Flattened_S_Ins_PF}
Let $\mathcal{S}$ be a multiset whose elements are in $[n+1]$ and let $|\mathcal{S}|$ denote the cardinality of $\mathcal{S}$ with multiplicity. A \textit{ $ \mathcal{S} $-insertion flattened parking function} is defined as a $\mathcal{S}$-insertion parking function that is flattened. These $ \mathcal{S} $-insertion flattened parking functions are contained in the subset $\textnormal{flat} (\mathcal{PF}_n (\mathcal{S})) \subset \mathcal{PF}_n(\mathcal{S}) $. 
Let $\textnormal{flat}_k (\mathcal{PF}_n (\mathcal{S}))$ denote the subset of $\textnormal{flat}(\mathcal{PF}_n (\mathcal{S}))$ where each parking function has exactly $ k $ runs.
\end{definition}

By Theorem \ref{thm:insertionAreParkingFxn} we have shown that $ \mathcal{PF}_n (\mathcal{S}) \subset \textnormal{PF}_{n+|S|}$, and hence $ \textnormal{flat} (\mathcal{PF}_n (\mathcal{S}))\subset\textnormal{flat} (\textnormal{PF}_{n+|S|})$. 

Although our main focus in subsequent sections is on $\mathcal{S}$-insertion flattened parking functions for the multiset $\mathcal{S}$ consisting of all ones, the following results provide interesting bijections between certain sets of  $\mathcal{S}$-insertion flattened parking functions.

\begin{theorem}
Let $\mathcal{S}$ be a multiset with elements in $[n]\setminus\{1\}$ and let $\mathcal{S}'=\{s-1:s\in\mathcal{S}\}\cup\{1\}$. 
Then, the sets $\textnormal{flat}(\mathcal{PF}_n (\mathcal{S}))$ and $\textnormal{flat}(\mathcal{PF}_{n-1} (\mathcal{S}'))$ are in bijection.
\end{theorem}

\begin{proof}
Let $|\mathcal{S}|=m$ and note that $|\mathcal{S}'|=m+1$. Let $a=a_1a_2\cdots a_{n+m}\in \textnormal{flat}(\mathcal{PF}_n (\mathcal{S}))$. Define the map $\psi:\textnormal{flat}(\mathcal{PF}_n (\mathcal{S})) \rightarrow 
\textnormal{flat}(\mathcal{PF}_{n-1} (\mathcal{S}'))$, by $\psi(a)=b_1b_2\cdots b_{n+m}$,  where 
\[b_i=\begin{cases}a_1, & \text{ if $i=1$;}\\a_i-1, & \text{ if $1<i\leq n+m$}.\end{cases}\]

We note that the function is well defined, as $\psi(a)$ maps to a unique parking function. We also note that our function definition preserves flattenedness, since descents $a_i \geq a_{i+1}$ implies that $b_{i}\geq b_{i+1}$ is also a descent.

To show the map $\psi$ is onto, let $b=b_1b_2\cdots b_{n+m}\in \textnormal{flat}(\mathcal{PF}_{n-1} (\mathcal{S}'))$. Observe that $b$ is a parking function built from a permutation on $[n-1]$, with all $m+1$ elements of $\mathcal{S}'$ inserted into the permutation. 
Let $b'=b_1'b_2'\cdots b_{n+m}'$ be such that
\[b'_i=\begin{cases}b_1, & \text{ if $i=1$;}\\b_i+1, & \text{ if $1<i\leq n+m$}.\end{cases}\]
Note that this new parking function, $b'$, contains only one one, namely $b_1=b_i'$. 
By construction of $\mathcal{S}'$, and the fact that we now have at least one instance of all the numbers of $[n]$, $b'$ is an element of $\textnormal{flat}(\mathcal{PF}_n (\mathcal{S}))$. 
Now note that $\psi(b') = b$, as desired. 

To show the map $\psi$ is one-to-one suppose that $\psi(a) = \psi(c)=b_1b_2\cdots b_{n+m}$, for $a=a_1a_2\cdots a_{n+m}$ and $c=c_1c_2\cdots c_{n+m}$ in $\textnormal{flat}(\mathcal{PF}_{n} (\mathcal{S}))$. By definition of $\psi$ this means that $b_1=a_1=c_1$ and for all $1<i\leq n+m$, $b_i=a_i-1=c_i-1$, which implies that $a_i=c_i$. Thus $a_i=c_i$ for all $1\leq i\leq n+m$, establishing that $a=c$, as desired. 

Therefore, we have shown that $\psi$ is a bijection between the sets $\textnormal{flat}(\mathcal{PF}_n (\mathcal{S}))$ and $\textnormal{flat}(\mathcal{PF}_{n-1} (\mathcal{S}'))$, as claimed.
\end{proof}

\begin{theorem}\label{thm:n(n-1)Bijec}
Let $\mathcal{S}$ be a multiset with elements in $[n-2]$. 
Then the sets $\textnormal{flat}(\mathcal{PF}_n (\mathcal{S}\cup\{n\}))$ and $\textnormal{flat}(\mathcal{PF}_{n} (\mathcal{S}\cup\{n-1\}))$ are in bijection. 
\end{theorem}

\begin{proof}
Let $|\mathcal{S}| = m$, where $ s\in [n-2]$ for all $s\in \mathcal{S}$. Let $a=a_1a_2\cdots a_{n+m+1} \in \textnormal{flat}(\mathcal{PF}_n (\mathcal{S}\cup\{n-1\}))$. Note that $a$ contains the value $n-1$ exactly twice, and our approach is to replace one of these with the value $n$. To do this we define the map $\psi : \textnormal{flat}(\mathcal{PF}_n(\mathcal{S} \cup \{n-1\}))
\rightarrow \textnormal{flat}(\mathcal{PF}_n(\mathcal{S} \cup \{n\}))$ 
with 
$\psi(a)=b_1b_2\cdots b_{n+m+1}$ where
if $1\leq i<j<k\leq n+m+1$ and 
\begin{enumerate}
\item $a_i=n-1$, $a_j=n-1$, and $a_k=n$, then $b_s=a_s$ for all $s\neq j$, and $b_j=n$, 
\item $a_i=n-1$, $a_j=n$, and $a_k=n-1$, then $b_s=a_s$ for all $s\neq i$, and $b_i=n$,
\item $a_i=n$, $a_j=n-1$, and $a_k=n-1$, then $b_s=a_s$ for all $s\neq k$, and $b_k=n$.
\end{enumerate}

To show that the map $\psi$ is well defined, first 
note that the entries in $\psi(a)$ are those in the multiset $[n]\cup\mathcal{S}\cup\{n\}$. In $\psi(a)$ we replaced a single instance of $n-1$ in $a$ with the value $n$. This process keeps all weak ascents (and descents) in $a$ as weak ascents (and descents) in $\psi(a)$. This process also keeps the relative order of the leading value of the runs. Thus ensuring that $\psi(a)$ is in $\textnormal{flat}(\mathcal{PF}_n(\mathcal{S} \cup \{n\}))$.

We now construct the inverse. Let $b\in \textnormal{flat}(\mathcal{PF}_n(\mathcal{S} \cup \{n\}))$. Note that $b$ contains two instances of the value $n$ and one instance of the value $n-1$.
Define the map $\phi : \textnormal{flat}(\mathcal{PF}_n(\mathcal{S} \cup \{n\}))
\rightarrow \textnormal{flat}(\mathcal{PF}_n(\mathcal{S} \cup \{n-1\}))$ 
with 
$\phi(b)=a_1a_2\cdots a_{n+m+1}$ 
where if $1\leq i<j<k\leq n+m+1$ and 
\begin{enumerate}
\item[~]Case 1. $b_i=n-1$, $b_j=n$, and $b_k=n$, then $a_s=b_s$ for all $s\neq j$, and $a_j=n-1$, 
\item[~]Case 2. $b_i=n$, $b_j=n$, and $b_k=n-1$, then $a_s=b_s$ for all $s\neq i$, and $a_i=n-1$,
\item[~]Case 3. $b_i=n$, $b_j=n-1$, and $b_k=n$, then $a_s=b_s$ for all $s\neq k$, and $a_k=n-1$.
\end{enumerate}

To show that the map $\phi$ is well defined, first 
note that the entries in $\phi(b)$ are those in the multiset $[n]\cup\mathcal{S}\cup\{n-1\}$. In $\phi(b)$ we replaced a single instance of $n$ in $b$ with the value $n-1$. This process keeps all weak ascents (and descents) in $b$ as weak ascents (and descents) in $\phi(b)$. This process also keeps the relative order of the leading value of the runs. Thus ensuring that $\phi(b)$ is in $\textnormal{flat}(\mathcal{PF}_n(\mathcal{S} \cup \{n-1\}))$.

To conclude, note that $\phi$ is the inverse of $\psi$. This holds because for all $1\leq x\leq 3$, if $b$ is as in Case $x$, then $\phi(b)=a$ gives $a$ defined as in $(x)$ above, and $\psi(a)$ returns $b$, as desired.
\end{proof}

In general, we can use the reasoning in Theorem \ref{thm:n(n-1)Bijec} to produce the following bijection when there are only two runs.

\begin{theorem} \label{bijection2runs}
If $ 2 \leq \ell \leq n $, then the sets $ \textnormal{flat}_{2}(\mathcal{PF}_{n}(\{\ell-1\})) $ and $\textnormal{flat}_{2}(\mathcal{PF}_{n}(\{\ell\}))$ are in bijection.
\end{theorem}

\begin{proof}
Let $a=a_1a_2\cdots a_{n+1} \in \textnormal{flat}_2(\mathcal{PF}_n (\{\ell-1\}))$ for some $ 2 \leq \ell \leq n $. Note that $a$ contains the value $\ell-1$ exactly twice, and our approach is to replace one of these with the value $\ell$. To do this we define the map $\psi : \textnormal{flat}_2(\mathcal{PF}_n (\{\ell-1\}))
\rightarrow \textnormal{flat}_2(\mathcal{PF}_n (\{\ell\}))$ 
with 
$\psi(a)=b_1b_2\cdots b_{n+1}$ where
if $1\leq i<j<k\leq n+1$ and 
\begin{enumerate}
\item $a_i=\ell-1$, $a_j=\ell-1$, and $a_k=\ell$, then $b_s=a_s$ for all $s\neq j$, and $b_j=\ell$, 
\item $a_i=\ell-1$, $a_j=\ell$, and $a_k=\ell-1$, then $b_s=a_s$ for all $s\neq i$, and $b_i=\ell$,
\item $a_i=\ell$, $a_j=\ell-1$, and $a_k=\ell-1$, then $b_s=a_s$ for all $s\neq k$, and $b_k=\ell$.
\end{enumerate}

To show that the map $\psi$ is well-defined, first 
note that the entries in $\psi(a)$ are those in the multiset $[n]\cup\{\ell\}$. In $\psi(a)$ we replaced a single instance of $\ell-1$ in $a$ with the value $\ell$. 
By construction we know that the resulting list $\psi(a)$ is unique. 
We now claim that $\psi(a)$ has exactly two runs. 
We argue based on 
which case $a$ satisfies.

In item 1: replacing $a_j=\ell-1$ with the value $\ell$ means that, if $a_j$
 ended a run, then so does $b_j=\ell$.
 If  $a_j$ did not end a run, that means that 
 $a_j=\ell-1< a_{j+1}$ and $a_{j+1}\geq \ell$, which implies that replacing $a_j=\ell-1$ with $\ell$ also does not end the run. Thus, in case 1, the number of runs remains the same. 

In item 2: we are replacing $a_i=\ell-1$ with the value $\ell$. Note that $a$ begins with two runs. 
If $a_i=\ell-1$ and $a_j=\ell$ are not in the same run, then we currently have at least two runs. However, $a_k=\ell-1$ and $k>j$ implies that $a_k$ would not be in the same run as $a_j$. Hence, $a$ would have a minimum of three distinct runs, a contradiction. 
Thus, $a_i=\ell-1$ and $a_j=\ell$ must be in the same run, which implies they must be consecutive. That is, we know $j=i+1$. In this case, replacing the subword $(\ell-1)\ell$ with $\ell\ell$ keeps the number of runs at two runs, as desired.

In item 3:  we are replacing $a_k=\ell-1$ with the value $\ell$. Note that to begin $a$ had two runs. 
This means that $a_i=\ell$ is in the first run, while both instances of $\ell-1$, in this case $a_j$  and $a_k$ are in the same run, and the only way for this to be true is for them to be consecutive. 
Moreover, $a_{k+1}\geq \ell$ as the opposite inequality would imply there are at least three runs. Moreover, $a_{k+1}>\ell$ as the only instance of $\ell$ in $a$ appeared at index $i$. 
Hence in $a$ we replace the subword $(\ell-1)(\ell-1)$ with $(\ell-1)\ell$, which would not introduce a new run.

Therefore, under all assumptions, $\psi(a)$ results in a unique element in $\textnormal{flat}_2(\mathcal{PF}_n(\{\ell\}))$.

Therefore, this process keeps all weak ascents (and descents) in $a$ as weak ascents (and descents) in $\psi(a)$. This process also keeps the relative order of the leading value of the runs, ensuring that $\psi(a)$ is in $\textnormal{flat} _2(\mathcal{PF}_n( \{\ell\}))$.

We now construct the inverse. 
Let $b$ be in $ \textnormal{flat} _2(\mathcal{PF}_n( \{\ell\}))$. 
Note that $b$ contains two instances of the value ${\ell}$ and one instance of the value $\ell-1$.
Define the map $\phi : \textnormal{flat}_2(\mathcal{PF}_n(\{\ell\}))
\rightarrow \textnormal{flat}_2(\mathcal{PF}_n( \{\ell-1\}))$ 
with 
$\phi(b)=a_1a_2\cdots a_{n+1}$ 
where if $1\leq i<j<k\leq n+1$ and 
\begin{enumerate}
\item[~]Case 1. $b_i=\ell-1$, $b_j=\ell$, and $b_k=\ell$, then $a_s=b_s$ for all $s\neq j$, and $a_j=\ell-1$, 
\item[~]Case 2. $b_i=\ell$, $b_j=\ell$, and $b_k=\ell-1$, then $a_s=b_s$ for all $s\neq i$, and $a_i=\ell-1$,
\item[~]Case 3. $b_i=\ell$, $b_j=\ell-1$, and $b_k=\ell$, then $a_s=b_s$ for all $s\neq k$, and $a_k=\ell-1$.
\end{enumerate}

To show that the map $\phi$ is well-defined, first 
note that the entries in $\phi(b)$ are those in the multiset $[n]\cup\{\ell-1\}$. In $\phi(b)$ we replaced a single instance of $\ell$ in $b$ with the value $\ell-1$. As before, by construction, we know that $\phi(b)$ is unique. We now claim that $\phi(b)$ has exactly two runs. We argue based on which case $b$ satisfies.

In Case 1: suppose first that $b_j = \ell$ ends the first run. Then since $(\ell-1)$ is to the right, we must have the subword $(\ell-1)\ell~b_{j+1}$, where $b_{j+1}<(\ell -1)$, as there is only one instance of $\ell-1$ in $b$. This subword maps to $(\ell-1)(\ell-1)~b_{j+1}$, which preserves the run, and the descent in position $y$. 
Now suppose that $b_j$ is in the middle of a run. Then we have the subword $\ell~b_{j+1}$, where $b_{j+1}\geq \ell$. This subword maps to $(\ell-1)b_{j+1}$, which also preserves the run. Thus, in case 1, the number of runs remains the same.

In Case 2: we are replacing $b_{i} = \ell$ with the value $\ell-1$.  Note that $b$ has two runs. As before, if $b_i$ and $b_j=\ell$ are in separate runs, then we must have at least three runs in $b$. Thus, $b_i$ and $b_j$ must be in the same run, which implies that they are consecutive. That is, $j=i+1$. Furthermore, $b_{i-1}<\ell-1$, as the only occurrence of $\ell-1$ appears to the right of $b_j$. In this case, the subword $b_{i-1}\ell\ell$ maps to $b_{i-1}(\ell-1)\ell$, which preserves the run. 

In Case 3: we are replacing $b_k=\ell$ with the value $\ell-1$. Note that $b$ has two runs. This means that $b_j=\ell-1$ and $b_k$ must be in the second run, while $b_i$ is in the first run. Because $b_j=\ell-1$ and $b_k=\ell$ are in the same run, we have $k=j+1$. Moreover, $b_{k+1}>\ell$, as the only other instance of $\ell$ appears to the left of $b_k$. Hence the subword $(\ell-1)\ell b_{k+1}$ maps to $(\ell-1)(\ell-1)b_{k+1}$, which preserves the run.  

This process keeps all weak ascents (and descents) in $b$ as weak ascents (and descents) in $\phi(b)$. This process also keeps the relative order of the leading value of the runs, ensuring that $\phi(b)$ is in $\textnormal{flat}(\mathcal{PF}_n(\{\ell-1\}))$.

To conclude note that $\phi$ is the inverse of $\psi$. This holds as for all $1\leq x\leq 3$, if $b$ is as in Case $x$, then $\phi(b)=a$ gives $a$ defined as in ($x$) above, and $\psi(a)$ returns $b$, as desired.
\end{proof}

 We remark that 
 Theorem \ref{bijection2runs} is not true for all $2\leq \ell \leq n$ if the number of runs is not exactly two. 
For example, if $n=4$ and $\ell = 3$, then there are two elements in $\textnormal{flat}_3(\mathcal{PF}_{4}(\{2\}))$: $14232$ and $13242$. Using the map defined in the proof of Theorem \ref{bijection2runs}, we have $\psi(14232) = 14332$, which is not flattened. By restricting our map to only two runs, we know exactly where the  values $\ell-1$ and $\ell$ are with respect to each other, and our map preserves the number of runs. In Theorem \ref{thm:n(n-1)Bijec}, we did not have to worry about this restriction to the number of runs, as the values $n-1$ and $n$ were extreme values that allowed us to preserve all ascents and descents we started with.

We now present the following enumerative result.
\begin{theorem}\label{enum result}
If $n\geq 3$ and $\ell\geq 2$, then $|\textnormal{flat} _2(\mathcal{PF}_n(\{ \ell\}))|=\sum_{i=0}^{n-3} 2^{i} (n-1-i)$. 
\end{theorem}
The proof of Theorem \ref{enum result} follows from the bijection in Theorem \ref{bijection2runs} and the following technical results, which provide a connection to tableaux enumerations. 

\begin{lemma}[\seqnum{A077802}]\label{prop:help_from_A}
The sum of products of parts increased by one in hook partitions~is
\begin{equation*}
\sum_{i=0}^{m-1} 2^i (m+1-i).
\end{equation*}
\end{lemma}
\begin{proof}
A hook partition on $m$ boxes has $i$ rows of exactly one box, such that $0 \leq i \leq m-1$, and one row of length $m-i$. In increasing each part by one, we add one box to each row, to get $i$ rows of length two and a single row of length $m-i+1$. Then, the product of the lengths of the rows is given by $2^i (m-i+1)$. Finally, we sum the products over all possible hook partitions of size $m$. Note that each hook partition corresponds to exactly one value of $i$, giving the sum:
$\sum_{i=0}^{m-1} 2^i (m+1-i)$.
\end{proof}

\begin{lemma}\label{thm:hooknumbersbijection}
If $n\geq 3$, then $|\textnormal{flat} _2(\mathcal{PF}_n(\{ 2\}))|=\sum_{i=0}^{n-3} 2^{i} (n-1-i)$.
\end{lemma}

\begin{proof}
We prove this by induction on $ m =n-2$.
If $ m = 1 $, then
\begin{equation*}
\sum_{i=0}^{m-1} 2^i (m+1-i) = 2.
\end{equation*}
We note that the set $\textnormal{flat} (\mathcal{PF}_3(\{ 2\})) = \{ 1223,1232, 1322\}$ is formed by inserting a 2 into a permutation of $\mathfrak{S}_3$, and then only selecting the flattened parking functions. We are only interested in the subset $\textnormal{flat}_2 (\mathcal{PF}_3(\{ 2\}))$, with two runs. In this case, there are two such elements: $1232$ and $1322$.  So indeed, $ \vert\textnormal{flat} _2(\mathcal{PF}_3(\{ 2\}))\vert = 2 $.

Suppose that $ \sum_{i=0}^{m-1} 2^i (m+1-i) = \vert \textnormal{flat}_2(\mathcal{PF}_n(\{ 2\})) \vert $, where $n=m+2$.
We now establish that the result hold for $m+1$. Note that we can get some of the elements in $ \textnormal{flat}_2(\mathcal{PF}_{n+1}(\{ 2\})) $ by inserting $ n + 1 $ into the elements in $ \textnormal{flat}_2(\mathcal{PF}_{n}(\{ 2\})) $. 
Since we have two runs, we can insert $n+1$ to the end of the first run or to the end of the second run. By induction, this gives $2\sum_{i=0}^{m-1} 2^i (m+1-i)$ parking functions. 
We can also get elements of $\vert \textnormal{flat}_{n+1}(\mathcal{PF}_n\left(\{ 2\}\right))\vert$ by inserting $n+1$ into an element of  $\textnormal{flat}_(\mathcal{PF}_n(\{ 2\}))$ with a single run, which increases the number of runs by one. 
Note that we only have one element with a single run, namely, when all elements are in weakly increasing order. 
In this case, we can insert $n+1$ into any spot other than the end in order to increase the number of runs, as a descent always occurs after the inserted $ n + 1 $. 
Note that we have $n = m + 2$ such spots in which to insert the value $n+1$, hence we have constructed $m+2 + 2\sum_{i=0}^{m-1} 2^i (m+1-i)$ elements in $ \textnormal{flat}_2(\mathcal{PF}_{n+1}(\{ 2\})) $. 
Now note that
\begin{align*}
m+2+2\sum_{i=0}^{m-1} 2^{i} (m+1-i) & = m+2+ \sum_{i=0}^{m-1} 2^{i+1} (m+2-(i+1)) \\
& = m + 2 + \sum_{i=1}^{m} 2^{i} (m+2-i) \\
& = \sum_{i=0}^{m} 2^{i} (m+2-i).
\end{align*}
Thus, $|\textnormal{flat}_2(\mathcal{PF}_{n}(\{ 2\}))|= \sum_{i=0}^{m-1} 2^{i} (m+1-i) $.
\end{proof}

We can now proceed to our main object of study,  $\mathcal{S}$-insertion flattened parking functions in the special case of $\mathcal{S}$ being the multiset consisting of all ones.

\section{\texorpdfstring{$ \mathbf{1}_{r} $}{Lg}-insertion flattened parking functions}\label{Sec:OnlyOneOne}
In this section, we let $ \mathcal{S} \coloneqq \mathbf{1}_{r} =\{1,\ldots,1\} $ be the multiset containing $r\geq 1$ ones. We then study families of $ \textnormal{flat} (\mathcal{PF}_{n} (\mathbf{1}_{r})) $ for $r\geq 1$.
To highlight how our results generalize the work of Nabawanda, Rakotondrajao and Bamunoba in \cite{NaRaBa}, we begin with Subsection~\ref{sec:single1} which contains the results in the special case of $r=1$. We then further generalize to $r\geq 2 $ in Subsection~\ref{sec:manyOnes}.

\subsection{The case of \texorpdfstring{$r=1$}{Lg}}\label{sec:single1}
In this section, we study the set $ \textnormal{flat} (\mathcal{PF}_n(\{1\})) $, which is the case of $\textnormal{flat} (\mathcal{PF}_n(\mathbf{1}_{r}))$, with $r=1$. 
Henceforth, we call this family of parking functions $ 1 $-insertion flattened parking functions. 

Our work is motivated by the results of Nabawanda, Rakotondrajao and Bamunoba in~\cite{NaRaBa}. In their work, they study $\mathcal{F}_{n,k}$, which is the set of flattened partitions with $k$ runs. 
They denote the cardinality of this set as $f_{n,k}$.
In what follows we recall their recursions for $f_{n,k}$, and provide our generalizations thereafter.

\subsubsection{Recursions for \texorpdfstring{$ \textnormal{flat} (\mathcal{PF}_n(\{1\}))$}{Lg}}
For sets of  $ \mathcal{S} $-insertion flattened parking functions, we let $ f(\mathcal{S}; n) $ denote the cardinality of $ \textnormal{flat} (\mathcal{PF}_n (\mathcal{S})) $, and 
    $f(\mathcal{S};n,k) $ denote the cardinality of $ \textnormal{flat} _k(\mathcal{PF}_n(\mathcal{S})) $.
We begin by recalling the recurrence for the cardinality of flattened partitions of length $n$ with $k$ runs as given by Nabawanda et al.\ \cite[Thm.\ 1, p.\ 4]{NaRaBa}. 

\begin{theorem}\label{thm:1inNAetal}
For all integers $ n,k $ such that $ 2 \leq k < n $, the numbers $f_{n,k}$ of flattened partitions of length $ n $ with $ k $ runs satisfy the following recurrence relation:
\begin{equation*}
f_{n,k} = \sum_{m=1}^{n-2} \left (\binom{n-1}{m}-1 \right)f_{m,k-1}.
\end{equation*}
\end{theorem}

The following result constructs elements of $\textnormal{flat} _k(\mathcal{PF}_n(\{1\}))$ using a similar argument as that in the proof of Theorem \ref{thm:1inNAetal}. 

\begin{theorem}\label{thm:thm1with1one}
For all integers $ n $ and $ k $ such that $ 2 \leq k < n $, the numbers $ f(\{1\}; n,k) $ of $1$-insertion flattened parking functions with k runs satisfy the following recurrence relation:
\begin{equation}\label{eq:ofthm1}
f(\{1\}; n,k) = \sum_{m=1}^{n-1} \left(\binom{n}{m} - 1\right) \cdot
f_{m, k-1}.
\end{equation}
\end{theorem}

\begin{proof}
We begin constructing our  $1$-insertion flattened parking function  $ \gamma $ by building the first run. 
Recall that every flattened parking function starts with a one and let $\gamma=1\alpha\beta$ where $1\alpha$ denotes the first run and $\beta$ contains the remaining $k-1$ runs of $\gamma$. Note that $\alpha\beta$ is a permutation of length $n$, and since $k\geq 2$, we know that $\beta$ cannot be empty.

Assume that $ \alpha $ is of length $ m $, where $m$ satisfies $ 1 \leq m \leq n-1$. This ensures $\beta$ is not empty. 
For a given $ m $, there are $\binom{n}{m} -1 $ options for the elements in the first run, namely the entries in $\alpha$. 
Note that the ``-1'' in our count comes from the fact that if $\alpha=123\cdots m$, then $1\alpha$ would not be the entirety of the first run as the number beginning $\beta$ would be larger than $m$ and hence would be a part of the run $1\alpha$, contradicting the assumption that $\alpha$ has length $m$. 
Moreover, all other choices for the entries in $\alpha$ result in a unique first run where those elements appear in increasing order. 
Thus, there are $\binom{n}{m}-1$ options for constructing the first run $1\alpha$.

Given the same value of $ m $ as before, we then construct the remaining $ k - 1 $ runs of $\gamma$, which are the entries in $\beta$. 
We know $\beta$ must be flattened for the sequence $1 \alpha \beta$ to be flattened, must contain $k-1$ runs, and has length $n-m$. 
We claim that there are $f_{n-m, k-1}$ choices for $\beta$. 
To see this, note that for each flattened partition with $k-1$ runs, we simply replace the set  $\{1,2,3,\ldots, n-m\}$ with the set of $n-m$ numbers making up $\beta$ and put them in the same relative order so as to create $k-1$ runs that are flattened. 
Hence, each flattened partition of length $n-m$ with $k-1$ runs gives a unique possible $\beta$. 

Finally, observe that our first run has minimum length 2 and maximum length $n$, in order to guarantee that we have at least two runs. 
So, this process holds for $1 \leq m \leq n-1$, giving us our final summation of $f(\{1\}; n, k) = \sum_{m=1}^{n-1} \left(\binom{n}{m} -1\right) f_{n-m,k-1}$.
To conclude, we note that
\begin{align*}
    f(\{1\}; n, k) &= \sum_{m=1}^{n-1} \left(\binom{n}{m} -1\right) f_{n-m,k-1}\\
    &= \sum_{m=1}^{n-1} \left(\binom{n}{n-m} -1\right) f_{n-m,k-1}\\
    &= \sum_{m=1}^{n-1} \left(\binom{n}{m} -1\right) f_{m,k-1},
\end{align*}
where the last equality follows from reindexing the sum. 
\end{proof}

Nabawanda et al.\ \cite[Thm.\ 10, p.\ 6]{NaRaBa}, constructed the following two term recursion.

\begin{theorem}\label{thm:10inNAetal}
For all integers $ n,k $ such that $ 1 \leq k < n $, the number of flattened partitions of length $ n $ with $ k $ runs satisfies the following recursion:
\begin{equation*}
f_{n,k} = kf_{n-1, k}+(n-2)f_{n-2, k-1}.
\end{equation*}
\end{theorem}

The following is a generalization of Theorem \ref{thm:10inNAetal} from flattened partitions to the set  $\textnormal{flat} _k(\mathcal{PF}_n(\{1\}))$. 
This result allows us to consider the patterns present in a parking function that contain the largest letter $n+1$.

\begin{theorem}\label{thm:1insPFRec}
For all integers $ n,k $ such that $ 1 \leq k < n $, the number of  $1$-insertion flattened parking functions with $k$ runs satisfies the recurrence relation
\begin{equation*}
f(\{1\};n+1,k) = k \cdot f(\{1\};n,k) + (n-1) \cdot f(\{1\};n-1,k-1) + f_{n,k-1}.
\end{equation*}
\end{theorem}

Before we proceed with the proof of Theorem \ref{thm:1insPFRec} we give some definitions and an initial result which provides a set partition of $\textnormal{flat} _k(\mathcal{PF}_{n+1}(\{1\}))$. We follow that with an example to illustrate the proof of Theorem~\ref{thm:1insPFRec}.
Recall that a word $\alpha$ is said to contain a subword $w$ if there exists consecutive entries in $\alpha$ whose values are precisely the values in $w$ in the same order. For example, $\alpha=12315647$ contains the subword $315$ but not the subword $213$.
We remark that Nabawanda et al.\ \cite{NaRaBa} utilized the idea of patterns to describe these subwords. For simplicity we omit this technicality as subwords are enough to present our results. 

\begin{proposition}\label{prop:disjointunion}
The set $\textnormal{flat} _k(\mathcal{PF}_{n+1}(\{1\}))$ can be partitioned (as a set) into the following four subsets:
\begin{enumerate}
    \item elements ending with $ n+1 $,
    \item elements beginning with the subword $1(n+1)1$,
    \item elements containing the subword $a(n+1)b$ with $a<b$,
    \item elements containing the subword $a(n+1)b$ with $a>b$.
\end{enumerate}
\end{proposition}

\begin{proof}
Let $\alpha\in \textnormal{flat} _k(\mathcal{PF}_{n+1}(\{1\}))$ and note that $\alpha$ has exactly one instance of the element $n+1$. 
It suffices to consider where the value $n+1$ appears in $\alpha$. 
First note that since $\alpha$ is flattened, then $n+1$ cannot appear at the start of $\alpha$. 
Thus, $\alpha$ contains $n+1$ either at the end or between two other numbers.
If $n+1$ appears at the end, then $\alpha$ is one of the elements of the subset of case 1. 
If $n+1$ appears between two other numbers, then $\alpha$ contains the subword $a(n+1)b$. In which case the following can occur: 
$a=b$, or $a<b$, or $a>b$. If $a=b$, then since $\alpha\in \textnormal{flat} _k(\mathcal{PF}_{n+1}(\{1\}))$, the only repeated value is one, and so $a=b=1$, which implies that $\alpha$ begins with the subword $1(n+1)1$. Hence, $\alpha$ is an element described in case 2.
If $a<b$ or $a>b$, then $\alpha$ is an element described in case 3 or case 4, respectively.
\end{proof}

Below we provide an example to illustrate the set partition in Proposition \ref{prop:disjointunion} and the enumeration in Theorem \ref{thm:1insPFRec}.  
\begin{example}
Let $n+1=4$ and $k=3$. By Proposition \ref{prop:disjointunion}
we note that 
\begin{enumerate}
    \item[(1)] there are no elements in $\textnormal{flat} _3(\mathcal{PF}_{4}(\{1\}))$ ending with $4$, 
    \item[(2)] there is one element in $\textnormal{flat} _3(\mathcal{PF}_{4}(\{1\}))$ beginning with the subword $141$: $14132$,
    \item[(3)] there are two elements in $\textnormal{flat} _3(\mathcal{PF}_{4}(\{1\}))$ containing the subword $a4b$ with $a<b$ are: $13142$ and $12143$    \item[(4)] there are no elements in $\textnormal{flat} _3(\mathcal{PF}_{4}(\{1\}))$ containing the subword $a4b$ with $a>b$. 
\end{enumerate}
Note that the number of elements from cases (1) and (4) is zero which is exactly the number $f(\{1\};3,3) $ as there are no  $1$-insertion flattened parking functions of length four with three runs. 
The number of elements in case (2) is the number $f_{3,2}$, which is a permutation of length three with two runs, namely the permutation $132$.
Lastly, the number of elements in case (3) is the number $2\cdot f(\{1\};2,2)$, since $f(\{1\};2,2)=1$, as the only  $1$-insertion flattened parking function of length three with two runs is $121$. 
This shows that 
Theorem \ref{thm:1insPFRec}
\[f(\{1\}; 4,3) = f(\{1\};3,3) + 2 \cdot f(\{1\};2,2) + f_{3,2}=0+2\cdot 1+1=3.\]
\end{example}

The following proposition utilizes techniques of Nabawanda et al.\ \cite[Thm.\ 9, p.\ 6]{NaRaBa}.

\begin{proposition}\label{prop:2n1}
The subset of $\textnormal{flat} _k(\mathcal{PF}_{n+1}(\{1\}))$ containing elements with the subword $a(n+1)b$ where $a>b$ or $n+1$ at the end of the parking function is enumerated by 
\begin{equation*}
k \cdot f(\{1\};n,k)
\end{equation*}
\end{proposition}

\begin{proof}
We start with $ 1 $-insertion parking functions of length $ n $ which have cardinality \[ f (\{1\};n,k) .\] To get the subword $a(n+1)b$ where $a>b$, we must place $n+1$ where it will not increase the number of runs. Thus, $ n+1 $ can be placed at the end of each run, including at the very end of the parking function, giving us $ k $ options for placement, and accounting for cases 1 and 4.
\end{proof}

The following proposition utilizes techniques of Beyene and Mantaci \cite[Lemma 25 and Proposition 26]{FBRM}.

\begin{proposition}\label{prop:1n2}
The subset of $\textnormal{flat} _k(\mathcal{PF}_{n+1}(\{1\}))$ containing elements with the subword $a(n+1)b$ with $a< b$ is enumerated by 
\begin{equation*}
(n-1) \cdot f(\{1\};n-1,k-1).
\end{equation*}
\end{proposition}

\begin{proof}
This proof is constructive. We begin by letting $ \pi $ be a parking function such that 
$ \pi \in
\textnormal{flat} _{k-1}(\mathcal{PF}_{n-1}(\{1\}))$. 
From $\pi$, we construct $\sigma \in \textnormal{flat} _k(\mathcal{PF}_{n+1}(\{1\}))$ such that $\sigma$ contains the subword $a(n+1)b$ with $1\leq a<b\leq n-1$. 

To construct $\sigma$ from $\pi$, we do the following procedure:

\begin{enumerate}
\item Fix a value $1\leq i \leq n-1$. 
\item Construct $\pi'$ from $\pi$ by taking every element in $\pi$ that is greater than $i$ and adding one to each of those elements.
\item Then, insert the subword $ (n+1)(i+1) $ into $ \pi' $ after the rightmost element of $ \{1, \ldots,i\} $ in $\pi'$. 
\item Call the resulting word $\sigma$.
\end{enumerate}

We claim that $\sigma\in \textnormal{flat} _k(\mathcal{PF}_{n+1}(\{1\}))$.
First note that from step $3$, the length of  $\sigma$ is $n+2$.
Additionally, because $\pi'$ does not contain $i+1$ or $n+1$, we have exactly one instance of every letter $j$, where $2\leq j\leq n+1$. 
Hence, $\sigma$ consists of the letters in the multiset $[n+1]\cup \{1\}$ as desired.
Note that by construction of $\sigma$, wherever we inserted $(n+1)(i+1)$, the value to the left of $(n+1)(i+1)$ would never end a run, and hence, since $n+1>i+1$, this insertion creates an additional run.
Lastly, we are guaranteed that $\sigma$ is flattened, since $\pi\in \textnormal{flat} _{k-1}(\mathcal{PF}_{n-1}(\{1\}))$ and by step $2$ and $3$ we inserted $(n+1)(i+1)$ after all elements smaller than $ i + 1 $ and hence the number of runs to the left of the inserted $(n+1)(i+1)$ remains fixed. Moreover, as all of the elements to the right of the newly inserted $(n+1)(i+1)$ are larger than $i+1$, we know that the number of those runs also remained fixed. Thus, the total number of runs in $\sigma$ is equal to one more than the number of runs in $\pi$, which is $(k-1)+1=k$.
This establishes that $\sigma\in \textnormal{flat} _k(\mathcal{PF}_{n+1}(\{1\}))$.

Furthermore, this construction can be reversed. 
Suppose that $\sigma \in \textnormal{flat} _k(\mathcal{PF}_{n+1}(\{1\}))$ containing the subword $a(n+1)b$ with $a<b$. Since $\sigma$ is flattened, removing the subword $(n+1)b$ from $\sigma$ yields a new word $\sigma'$ of length $n$ that is missing the value $b$, contains $n$, and has $k-1$ runs. 
Take all the letters in $\sigma'$ that are greater than $b$ and subtract one from each letter. 
This gives a word $\pi$ that is in $\textnormal{flat} _{k-1}(\mathcal{PF}_{n-1}(\{1\}))$. 

Given that we fixed the value $ i $ satisfying $1\leq i\leq n-1$, we have $ n -1 $ options for this integer. Hence, for each choice, we have created $ f(\{1\};n-1,k-1) $ options for $ \pi $, yielding a total of $ (n-1) \cdot f(\{1\};n-1,k-1) $ parking functions  in $\textnormal{flat} _k(\mathcal{PF}_{n+1}(\{1\}))$.
\end{proof}

Finally, the following proposition is the piece that is unique to our parking functions.

\begin{proposition}\label{prop:1n1}
The subset of $ \textnormal{flat} _k(\mathcal{PF}_{n+1}(\{1\})) $ containing elements with the subword $1(n+1)1$ is enumerated by
\begin{equation*}
f_{n,k-1}.
\end{equation*}
\end{proposition}

\begin{proof}
Let $ \sigma \in F_{n,k-1} $. 
In order to get a parking function in $ \textnormal{flat} _k(\mathcal{PF}_{n+1}(\{1\})) $ that contains the subword $1(n+1)1$, we insert the subword $1(n+1)$ to the start of $ \sigma $. 
Since this is the only spot we can insert $1(n+1)$ to contain the desired pattern and remain flat, we have that there are $ f_{n ,k - 1} $ parking functions obtained from this process. 
\end{proof}

We now return to Theorem \ref{thm:1insPFRec}. 

\begin{proof}[Proof of Theorem \ref{thm:1insPFRec}]
In Proposition \ref{prop:disjointunion}, we showed that our four subsets partition \[ \textnormal{flat}_{k} (\mathcal{PF}_{n+1}(\{1\})) .\] Using Propositions \ref{prop:2n1}, \ref{prop:1n2}, and \ref{prop:1n1}, we have
\begin{equation*}
f(\{1\}; n+1,k) = k \cdot f(\{1\}; n,k) + (n-1) \cdot f(\{1\}; n-1,k-1) + f_{n,k-1}.\qedhere
\end{equation*}
\end{proof}

Finally, Nabawanda et al.\ \cite[Thm.\ 12, p.\ 7]{NaRaBa} gave one more recursion for $f_{n,k}$.

\begin{theorem}\label{thm:12inNAetal}
For all integers $ n,k $ such that $ 1 \leq k \leq n $, the number of flattened partitions of length $ n+2 $ with $ k $ runs satisfies the following recursion:
\begin{equation*}
f_{n+2,k} =f_{n+1,k} +\sum_{i=1}^n \binom{n}{i} f_{n+1-i, k-1}.
\end{equation*}
\end{theorem}

Our generalization of this recursion appears below. This recursion allows us to build a $1$-insertion parking function of length $n+2$ by considering whether or not the two ones appear in the same run or in distinct runs.

\begin{theorem}\label{thm:Ones_gen_func_recurs}

For all integers $ n,k $ such that $ 1 \leq k < n $, the number of $1$-insertion flattened parking functions with $k$ runs satisfies the recurrence relation
\begin{equation*}
f(\{1\};n+1,k) = f_{n+1,k} + \sum_{i=1}^n \binom{n}{i}
f_{n-i+1,k-1}.
\end{equation*}
\end{theorem}

\begin{proof}
We construct a  $1$-insertion flattened parking function $\pi$ over $[n+1]$ with $k$ runs.
We consider the following two possibilities to enumerate 
$f(\{1\};n+1,k)$:
\begin{enumerate}
\item Suppose both ones are in the same run and $ \pi $ is of the form $ \pi = 1 \tau $, where $ \tau $ is a flattened partition of length $ n+1$ that begins with a one. Remove the first one, and note that what remains is a flattened partition of length $n+1$ with $k$ runs. 
\item Suppose the two ones are in different runs. Let the first run have $i+1$ terms including the first one. Since the one is fixed, we have $i$ terms which can be chosen from the set $\{2,3, \ldots, n+1\}$. There are $\binom{n}{i}$ ways to do this. 
The remaining $k-1$ runs have $n+2 - (i+1) = n - i + 1$ elements including the second instance of a one. Since the one is fixed, note that the remaining terms are all distinct values and behave exactly like a permutation in the set $F_{n-i+1, k-1}$, so they are enumerated by $f_{n-i+1,k-1}$. Thus, the total number of  $1$-insertion flattened parking functions with the two ones in separate runs is $\sum_{i=1}^{n} \left(\binom{n}{i} f_{n-i+1,k-1}\right)$.
\end{enumerate}
As the cases are disjoint, our count is given by
\begin{equation*}
f(\{1\};n+1,k) = f_{n+1,k} + 
\sum_{i=1}^{n} \binom{n}{i}f_{n-i+1, k-1}.\qedhere
\end{equation*}
\end{proof}

Theorem \ref{thm:Ones_gen_func_recurs} implies that $ \textnormal{flat} _k(\mathcal{PF}_{n+1}(\{1\})) $ and $\mathcal{F}_{n+2,k}$ are equinumerous.
We use the recurrence relation provided in Theorem \ref{thm:Ones_gen_func_recurs} when we study generating functions in Section \ref{sec:GenFunc}.

We recall that the Bell numbers (\seqnum{A000110}), denoted $B(n)$, counts the number of set partitions of $ [n] $ and has the exact same recurrence relation as that of $f(\{1\};n,k)=f_{n+1, k}$. This immediately implies the following.

\begin{corollary}\label{thm:BellNumbers1Ins} 
The cardinality of the set $\textnormal{flat}(\mathcal{PF}_{n}(\{1\}))$ is the Bell number $B(n)$. 
\end{corollary}

Also we recall that the Eulerian numbers, denoted $ A(n,1) $, counts the number of permutations of length $ n $ with one descent. This is equivalent to the number of permutations of length $ n $ with tqo runs. This in turn implies the following.

\begin{corollary}\label{thm:EulerianNums} 
The cardinality of the set $\textnormal{flat} _2(\mathcal{PF}_n(\{1\}))$ is the Eulerian number $A(n,1)$.
\end{corollary}

Next we provide a more general result for the case when $k>2$ as the values of $f(\{1\};n+1,k)$ coincides with the triangle in \seqnum{A124324}. We remark that this result was not noted 
by Nabawanda et al.\ \cite{NaRaBa}, though their data corresponds with ours.

\begin{definition}
Let $T(n,k)$ denote the number of ways to partition the set $ [n] $ such that there are exactly $k$ sets in the partition with at least two elements.
\end{definition}

For example, if $ n = 3 $ and $ k = 1 $, the set partitions\footnote{As is standard, when writing a set partition  of $[n]$ we use the notation ``/'' to separate the subsets used and we refer to each subset as a block.} of $[3]$ counted by $T(3,1)$ are  
$123$, $1/23$, $2/13$, and $3/12$, which are the only ways to partition $ [3] $ into subsets where exactly one of the subsets has two or more elements. Thus, we get $T(3,1)=4$.

\begin{proposition}\label{prop:k+1Runs_to_kBlocks}
If $ n \geq 1 $ and $ k \geq 1 $, then
$
f(\{1\};n,k+1) = T(n,k)
$.
\end{proposition}

\begin{proof}
We utilize the bijective map, $f$, given by Beyene and Mantaci \cite[Proposition 11]{FBRM} and Nabawanda et al.\ \cite[Proposition 16]{NaRaBa}. This map allows us to take a set partition of $[n]$, and create a flattened partition of $[n+1]$. Because the sets $\mathcal{F}_{n+1}$ and $\textnormal{flat}(\mathcal{PF}_n(\{ 1\}))$ are in bijection, we use this map to go from set partitions of $[n]$ to $\textnormal{flat}(\mathcal{PF}_n(\{ 1\}))$. 
In what follows, we show that in fact the map $f$ takes set partitions of $[n]$ which have exactly $k$ blocks with size greater than or equal to two, and maps them bijectively to
flattened partitions with $k+1$ runs, which we have shown (Theorem \ref{thm:Ones_gen_func_recurs}) are in bijection with elements of $\textnormal{flat}(\mathcal{PF}_n(\{1\}))$.

To begin we let $\mathcal{B}=\mathcal{B}_1/\mathcal{B}_2/\cdots/\mathcal{B}_\ell$ be a set partition of $[n]$ where exactly $k$ of the blocks have size at least two and the remaining $\ell-k$ blocks consist of a single element. 
Moreover, we assume that $\min(\mathcal{B}_i)<\min (\mathcal{B}_{i+1}) $ for all $i\in[\ell-1]$.
For a block $\mathcal{B}_i$ with size at least two, we create a subword $\sigma_i$ as follows: if $\mathcal{B}_i=\{b_1,b_2,\ldots,b_m\}$ with $b_1<b_2<\cdots<b_m$, we ``cycle'' the elements so that they are listed in the order $b_2b_3\cdots b_m b_1=\sigma_i$. Note that when $\mathcal{B}_i$ consists of a single element we let the corresponding subword $\sigma_i$ simply be that value. 
In this way we construct a word from the set partition $\mathcal{B}$ through concatenating the subwords to form the word $\sigma_1\sigma_2\cdots\sigma_\ell$.
We then introduce a one at the start of the word $\sigma_1\sigma_2\cdots\sigma_\ell$ to create $a=1\sigma_1\sigma_2\cdots\sigma_\ell$. 

We claim $a$ is a $1$-insertion flattened parking function with $k+1$ runs. 
To prove this we begin by noting that in fact the entries in $a$ consist of the elements in the set $[n]\cup\{1\}$. 
Next we check the number of runs. 
First note that if $\mathcal{B}_i$ has size at least 2, then $\sigma_i$ introduces a descent at its penultimate letter/index. 
Hence there are at least $k$ descents, as there are $k$ blocks of size at least 2. 
To see that there are exactly $k$ descents, note that all remaining blocks have size one, and because the blocks are ordered based on their minimum values, they never introduce a new descent when creating $a$.
This means that $a$ has exactly $k$ descents and hence $k+1$ runs.
The result is flattened since we placed a one at the start of $a$, which ensures that the first run begins with a one, and since the cycling process along with the ordering of the blocks based on their minimums being in increasing order, ensures that the runs in $a$ satisfy that their leading values are in weakly increasing order. 
Thus, $a$ is in $\textnormal{flat}_{k+1}(\mathcal{PF}_n(\{1\}))$.

So far, we have shown that under this map, $f$, an arbitrary set partition with exactly $k$ blocks of size at least two produce an element in $\textnormal{flat}_{k+1}(\mathcal{PF}_n(\{1\}))$.

Now, suppose that we have an arbitrary element, $\alpha$ in $\textnormal{flat} (\mathcal{PF}_n(\{1\}))$. Suppose that, under $f^{-1}$, $\alpha$ maps to a set partition with exactly $k$ blocks of size at least two, call this corresponding set partition $\mathcal{B}$. We have shown that $f(\mathcal{B})$ is in $\textnormal{flat}_{k+1}(\mathcal{PF}_n(\{1\}))$, and since $f$ is bijective, we must have that $f(\mathcal{B}) = \alpha$. Therefore, $\alpha$ must have been an element in $\textnormal{flat}_{k+1}(\mathcal{PF}_n(\{1\}))$.
\end{proof}

\subsection{The case of \texorpdfstring{$r\geq 2$}{Lg}} \label{sec:manyOnes}
Throughout we let $ \mathbf{1}_{r} $ denote the multiset consisting of $r$ ones. In this section we extend our results to the case $ \textnormal{flat} (\mathcal{PF}_{n} (\mathbf{1}_{r})) $ for $r\geq 2$.

\subsubsection{Recursions for \texorpdfstring{$\textnormal{flat} (\mathcal{PF}_{n} (\mathbf{1}_r )) $}{Lg}}
The following is a generalization of the result of Nabawanda et al.\ \cite[Thm.\ 1, p.\ 4]{NaRaBa} and Theorem \ref{thm:thm1with1one}.

\begin{theorem} \label{Thm:rOnesRecurs}
For $2 \leq k < n+r $, the number of $ \mathbf{1}_{r} $-insertion flattened parking functions is enumerated by the following recursion:
\begin{equation}\label{eq:r1s_first_rec}
f(\mathbf{1}_{r};n,k) = \sum_{m=1}^{n-1}\sum_{i=1}^r \binom{n-1}{m-1} \cdot f(\mathbf{1}_{r-i};m,k-1) + \sum_{m=1}^{n-1}\left(\binom{n-1}{m-1}-1\right) \cdot f_{m,k-1}.
\end{equation}
\end{theorem}

\begin{proof}
To begin we remark that all elements in $\textnormal{flat} (\mathcal{PF}_{n} (\mathbf{1} _{r})) $ have exactly $r+1$ ones. 
To construct an element in $ \textnormal{flat} (\mathcal{PF}_{n} (\mathbf{1} _{r})) $, we fix $i$, the number of ones in the first run. Note that $i$ satisfies
$ 1 \leq i \leq r+1 $.

Our first case is when $1\leq  i \leq r $. 
In this case, there exists at least one $1$ and at most $r-i+1$ ones in the last $k-1$ runs. 
In order to be flattened, the first run has the form $\underbrace{ 1 1 1 \cdots 1 }_{i}\sigma $, where $ \sigma $ is a non-empty word containing $n-m$ values chosen from the set $\{2,3,4,\ldots,n\}$. Note that this implies that $1\leq m\leq n-1$.
In this case there are $\binom{n-1}{n-m}=\binom{n-1}{m-1}$ choices for the values in $ \sigma $, which must be arranged in increasing order so as to construct the first run.
Then, for a given $ \sigma $, we construct the remaining $k-1$ runs, by forming the word $ \tau $ from the values in the complement of the set $\{2,3,\ldots,n\}$ used to form $\sigma$ along with the remaining $r-i+1$ ones. 
This ensures that $\tau$ has length $m+r-i$.

To construct $\tau$, we take an element of 
$ \textnormal{flat}_{k-1} (\mathcal{PF}_{m} (\mathbf{1}_{r-i})) $, and arrange $ r-i $ ones, in addition to the $m$ elements from $\{2,3,4,\ldots,n\}$ that are not included in $ \sigma $.
Again note that this ensures the length of $\tau$ is $m+r-i$, as expected.
Each element $\alpha\in \textnormal{flat}_{k-1} (\mathcal{PF}_{m} (\mathbf{1}_{r-i})) $ gives a unique $\tau$, by arranging the values of $\tau$ in the same relative order as those in $\alpha$.
This
yields $ f (\mathbf{1}_{r-i};m,k-1) $ options for our last $ k - 1 $ runs which are formed by $\tau$. 
By varying $i$ from one to $r$ and $m$ from one to $n-1$ we arrive at the left-most sum on the right hand-side of Equation \eqref{eq:r1s_first_rec}.

The second case counts the number of $ 1 $-insertion parking functions with all ones in the first run.
We construct such a parking function in a similar fashion, where $ \sigma $ and $ \tau $ are arranged to get $\underbrace{ 1 1 1 1\cdots 1}_{r+1}  \sigma \tau $.
To construct a word $ \sigma $ of length $ n - m $, we selected $n-m$ elements of $ \{2,3,\ldots,n\} $, which we can do in $ \binom{n-1}{n-m} = \binom{n-1}{m-1} $ ways.
We then eliminate the single case where the values of $ n - m $ are in increasing order $234\cdots n-m+1$, in which case the smallest value of $ \tau $ would be greater than the greatest value of $ \sigma $, so $ \tau $ would not create a new run as required.
This gives a total of $\binom{n-1}{m-1}-1$ options for $\sigma$.
Next, we have $ m $ elements of $\{2,3,4,\ldots,n\}$ that are not used in $ \sigma $, which are used to create $
\tau$. 
Note that every
$\alpha\in \mathcal{F}_{m,k-1}$ 
gives a unique $\tau$, by arranging the values of $\tau$ in the same relative order as those in $\alpha$.
This
yields $ f _{m,k-1 }$ options for our last $ k - 1 $ runs which are formed by $\tau$. 
Thus, for each $m\in[n-1]$, there are $(\binom{n-1}{m-1}-1)f_{m,k-1}$ options for these $\mathbf{1}_r$-insertion parking functions in this case.
By varying $m$ from one to $n-1$ we arrive at the right-most sum on the right hand-side of Equation \eqref{eq:r1s_first_rec}.

As the two cases discussed are disjoint and account for all possibilities, their sum yields
\[f(\mathbf{1}_{r};n,k) = \sum_{m=1}^{n-1}\sum_{i=1}^r \binom{n-1}{m-1} \cdot f(\mathbf{1}_{r-i};m,k-1) + \sum_{m=1}^{n-1}\left(\binom{n-1}{m-1}-1\right) \cdot f_{m,k-1}.\qedhere\]
\end{proof}
The following is a generalization of the result of Nabawanda et al.\
\cite[Thm.\ 10, p.\ 6]{NaRaBa} and Theorem~\ref{thm:1insPFRec}.

\begin{theorem}\label{thm:R1insPFRec}
The cardinality of $ \textnormal{flat}_{k} (\mathcal{PF}_{n} (\mathbf{1}_{r})) $ satisfies the recurrence relation
\begin{equation*}
f (\mathbf{1}_{r};n+1,k) = 
k \cdot f (\mathbf{1}_{r};n,k) + 
(n - 1 ) \cdot f (\mathbf{1}_{r};n-1,k-1) + 
r \cdot f (\mathbf{1}_{r-1};n,k-1).
\end{equation*}

\end{theorem}

The proof of Theorem \ref{thm:R1insPFRec} is a direct consequence of following technical results, which are analogous as those in the $r=1$ case presented in the previous section. 

\begin{proposition}\label{prop:Rdisjointunion}
    The set $\textnormal{flat} _k(\mathcal{PF}_{n+1}(\mathbf{1}_r))$ can be partitioned (as a set) into the following four subsets:
\begin{enumerate}
    \item elements ending with $ n+1 $,
    \item elements beginning with the subword $1(n+1)1$,
    \item elements containing the subword $a(n+1)b$ with $a<b$,
    \item elements containing the subword $a(n+1)b$ with $a>b$.
\end{enumerate}
\end{proposition}

\begin{proof} This proof is identical to that of Proposition \ref{prop:disjointunion}.
\end{proof}

\begin{proposition}\label{prop:R2n1}
The subset of $ \textnormal{flat}_{k} (\mathcal{PF}_{n+1} (\mathbf{1}_{r})) $ containing the subword 
$
a(n+1)b$ where $a>b$ or $ n + 1 $ at the end of the parking function is enumerated by
\begin{equation*}
k \cdot f (\mathbf{1}_{r};n,k).
\end{equation*}
\end{proposition}

\begin{proof}
We start with $ \mathbf{1}_{r} $-insertion flattened parking functions of length $ n $ which have cardinality $ f (\mathbf{1}_{r};n,k) $. To get the subword $ a(n+1)b $, where $a>b$, we must place $ n +1 $ where it will not increase the number of runs. 
Thus, we can place $ n+1 $ at the end of each existing run (which includes being at the very end of the parking function), giving us $ k $ options for placement, and accounting for cases 1 and 4.
\end{proof}

\begin{proposition}\label{prop:R1n2}
The subset of $ \textnormal{flat}_{k} (\mathcal{PF}_{n+1} (\mathbf{1}_{r})) $ containing elements with the subword $ a(n+1)b $ with $a<b$ is enumerated by 
\begin{equation*}
(n-1) \cdot f (\mathbf{1}_{r};n-1,k-1).
\end{equation*}
\end{proposition}

\begin{proof}
This proof is constructive. We begin by letting $ \pi $ be a parking function such that 
$ \pi \in
\textnormal{flat} _{k-1}(\mathcal{PF}_{n-1}(\mathbf{1}_r))$. 
From $\pi$ we construct $\sigma \in \textnormal{flat} _k(\mathcal{PF}_{n+1}(\mathbf{1}_r))$ so that $\sigma$ contains the subword $a(n+1)b$ with $1\leq a<b\leq n-1$. 

To construct $\sigma$ from $\pi$ we do the following procedure:
\begin{enumerate}
\item Fix a value $1\leq i \leq n-1$. 
\item Construct $\pi'$ from $\pi$ by taking every element in $\pi$ that is greater than $i$ and adding one to each of those elements.
\item Then, insert the subword $ (n+1)(i+1) $ into $ \pi' $ after the rightmost element of $ \{1, \ldots,i\} $ in $\pi'$. 
\item Call the resulting word $\sigma$.
\end{enumerate}

We claim that $\sigma\in \textnormal{flat} _k(\mathcal{PF}_{n+1}(\mathbf{1}_r))$.
First note that from step $3$, the length of  $\sigma$ is $n+1+r$.
Additionally, because $\pi'$ does not contain $i+1$ or $n+1$, we have exactly one instance of every letter $j$, where $2\leq j\leq n+1$. 
Hence, $\sigma$ consists of the letters in the multiset $[n+1]\cup \mathbf{1}_r$ as desired.
Note that by construction of $\sigma$, wherever we inserted $(n+1)(i+1)$, the value to the left of $(n+1)(i+1)$ would never end a run, and hence, since $n+1>i+1$, this insertion creates an additional run.
Lastly, we are guaranteed that $\sigma$ is flattened, since $\pi\in \textnormal{flat} _{k-1}(\mathcal{PF}_{n-1}\left(\mathbf{1}_r\right))$ and by steps $2$ and $3$ we inserted $(n+1)(i+1)$ after all elements smaller than $ i + 1 $ and hence the number of runs to the left of the inserted $(n+1)(i+1)$ remains fixed. 
Moreover, as all of the elements to the right of the newly inserted $(n+1)(i+1)$ are larger than $i+1$, we know that the number of those runs also remained fixed. 
Thus, the total number of runs in $\sigma$ is equal to one more than the number of runs in $\pi$, which is $(k-1)+1=k$.
This establishes that $\sigma\in \textnormal{flat} _k(\mathcal{PF}_{n+1}\left(\mathbf{1}_r\right))$.

Furthermore, this construction can be reversed. 
Suppose that $\sigma \in \textnormal{flat} _k(\mathcal{PF}_{n+1}\left(\mathbf{1}_r\right))$ containing the subword $a(n+1)b$ with $a<b$. Since $\sigma$ is flattened, removing the subword $(n+1)b$ from $\sigma$ yields a new word $\sigma'$ of length $n-1+r$ that is missing the value $b$, contains $n$, and has $k-1$ runs. 
Take all the letters in $\sigma'$ that are greater than $b$, and subtract one from each letter. 
This gives a word $\pi$ that is in $\textnormal{flat} _{k-1}(\mathcal{PF}_{n-1}\left(\mathbf{1}_r\right))$. 

Given that we fixed the value $ i $ satisfying $1\leq i\leq n-1$, we have $ n -1 $ options for this integer. Hence, for each choice, we have created $ f\left(\mathbf{1}_r;n-1,k-1\right) $ options for $ \pi $, yielding a total of $ \left(n-1\right) \cdot f\left(\mathbf{1}_r;n-1,k-1\right) $ parking functions  in $\textnormal{flat} _k(\mathcal{PF}_{n+1}\left(\mathbf{1}_r\right))$.
\end{proof}

\begin{proposition}\label{prop:R1n1multiple}
The subset of $ \textnormal{flat}_{k} (\mathcal{PF}_{n+1} \left(\mathbf{1}_{r}\right)) $ containing elements with the subword $ 1(n+1)1 $ is enumerated by 
\begin{equation*}
r \cdot f \left(\mathbf{1}_{r-1};n,k-1\right).
\end{equation*}
\end{proposition}

\begin{proof}
Let $ \sigma \in \textnormal{flat}_{k-1} (\mathcal{PF}_{n} \left(\mathbf{1}_{r-1}\right))$. 
In order to get a parking function in $ \textnormal{flat} _k(\mathcal{PF}_{n+1}\left(\mathbf{1}_{r}\right)) $ that contains the subword $1(n+1)1$, we insert the subword $1(n+1)$ to the start of $ \sigma $. 
Since we can insert the subword $1(n+1)$ to the left of every one in $\sigma $ and remain flat, and since $\sigma$ contains $r$ ones, we have that there are $r \cdot f(\mathbf{1}_{r-1};n,k-1) $ parking functions obtained from this process. 
\end{proof}

We now return to Theorem \ref{thm:R1insPFRec}. 

\begin{proof}
In Proposition \ref{prop:Rdisjointunion}, we showed that our four subsets partition $ \textnormal{flat}_{k} (\mathcal{PF}_{n} \left(\mathbf{1}_{r}\right)) $. Using Propositions \ref{prop:R2n1}, \ref{prop:R1n2}, and \ref{prop:R1n1multiple}, we have
\begin{equation*}
f (\mathbf{1}_{r};n+1,k) = 
k \cdot f (\mathbf{1}_{r};n,k) + 
(n - 1 ) \cdot f (\mathbf{1}_{r};n-1,k-1) + 
r \cdot f (\mathbf{1}_{r-1};n,k-1).\qedhere
\end{equation*}
\end{proof}

The following is a generalization of the result of Nabawanda et al.\ \cite[Thm.\ 12, p.\ 7]{NaRaBa} and Theorem \ref{thm:Ones_gen_func_recurs}.

\begin{theorem}\label{thm:R_Ones_gen_func_recurs}
The cardinality of the set $ \textnormal{flat}_{k} (\mathcal{PF}_{n} \left(\mathbf{1}_{r} \right)) $ satisfies the following recursion:
\begin{equation*}
f \left(\mathbf{1}_{r};n,k\right) = f \left(\mathbf{1}_{r-1};n,k\right) +
\sum_{i=1}^{n-1} \binom{n-1}{i} \cdot f \left(\mathbf{1}_{r-1};n-i,k-1\right).
\end{equation*}
\end{theorem}

\begin{proof}
We construct an element $\pi \in \textnormal{flat}_{k} (\mathcal{PF}_{n} \left(\mathbf{1}_{r}\right)) $. 
We consider the following two possibilities to enumerate $ f \left(\mathbf{1}_{r};n,k\right) $. 

\begin{enumerate}
\item Suppose there is only a single one in the first run and let the length of the first run be $ i + 1 $. 
Then, the first run has a leading term of one and there are $ \binom{n-1}{i} $ options for the remaining elements in the first run. 
The remaining $k-1$ runs follow the pattern of an element from $ \textnormal{flat}_{k-1} (\mathcal{PF}_{n-i} (\mathbf{1}_{r-1})) $ which is enumerated by $ f (\mathbf{1}_{r-1};n-i,k-1)  $.
\item Suppose there is more than a single one in the first run, such that $ \pi $ is of the form $ 1 \tau $, where $ \tau $ is a flattened parking function of length $ n + r - 1 $. 
We remove the first instance of a one from $ \pi $ to get an element from $ \textnormal{flat}_{k} (\mathcal{PF}_{n} (\mathbf{1}_{r-1})) $, which is enumerated by $ f (\mathbf{1}_{r-1};n,k) $.
\end{enumerate}

Adding these two cases together gives
\begin{equation*}
f (\mathbf{1}_{r};n,k) = f (\mathbf{1}_{r-1};n,k) +
\sum_{i=1}^{n-1} \binom{n-1}{i} \cdot f (\mathbf{1}_{r-1};n-i-1,k-1) 
.\qedhere
\end{equation*}
\end{proof}

We now recall that the $r$-Bell numbers enumerate the set of partitions of $ [n+r] $ such that the first $ r $ numbers are in different subsets. Let $B(n,r)$ denote these numbers. 
We define the $(r,k)$-Bell numbers as follows.

\begin{definition}\label{def:k_r_bell}
Let $B_{k}(n,r)$ be the number of set of partitions of $ [n+r] $ satisfying the following conditions:
\begin{enumerate}
\item the first $ r $ numbers are in different subsets
\item there are exactly $k$ blocks with at least two elements.
\end{enumerate}
\end{definition}

Note that $\sum_{k=0}^{n+r}B_{k}(n,r)=B(n,r)$, with some of the terms in the sum being zero based on the values of $n,r,k$.
The following result proves that $\mathbf{1}_r$-insertion flattened parking functions with $k+1$ runs are enumerated by the the $(r,k)$-Bell numbers.

\begin{theorem} \label{rkBellNumbers}
The set $ \textnormal{flat}_{k+1} (\mathcal{PF}_{n+1}(\mathbf{1}_r)) $ is enumerated by $ B_{k}(n,r) $.
\end{theorem}

\begin{proof}
We want to prove a bijection between the set of partitions of $[n+r]$ with the first $r$ numbers in different subsets  and $ \textnormal{flat}_{k+1} (\mathcal{PF}_{n+1}(\mathbf{1}_r)) $. 
To begin let $\mathcal{B}=\mathcal{B}_1/\mathcal{B}_2/\cdots/\mathcal{B}_\ell$ 
be a set partition of $[n+r]$ where the first $r$ integers appear in different blocks. Hence, $\ell \geq r$.
Moreover, since the order of the blocks does not affect the count, we assume that $\min(\mathcal{B}_i)<\min (\mathcal{B}_{i+1}) $ for all $i\in[\ell-1]$.
Suppose there are $k$ blocks with size at least two and the remaining $\ell-k$ blocks consist of a single element.
For a block $\mathcal{B}_i$ with size at least 2, we create a subword $\sigma_i$ as follows: if $\mathcal{B}_i=\{b_1,b_2,\ldots,b_m\}$ with $b_1<b_2<\cdots<b_m$, we ``cycle'' the elements so that they are listed in the order $b_2b_3\cdots b_m b_1=\sigma_i$. Note that when $\mathcal{B}_i$ consists of a single element we let the corresponding subword $\sigma_i$ simply be that value. 
In this way we construct a word from the set partition $\mathcal{B}$ through concatenating the subwords to form the word $\sigma_1\sigma_2\cdots\sigma_\ell$.
We then introduce a one at the start of the word $\sigma_1\sigma_2\cdots\sigma_\ell$ to create $a=1\sigma_1\sigma_2\cdots\sigma_\ell$. 
In $\sigma_1\cdots\sigma_\ell$ replace the integers $1,2,\ldots, r$ with ones, so that now in $a$ there are exactly $r+1$ ones and the remaining numbers are strictly larger than $r$. 
We then subtract $ r - 1 $ from all other terms  creating a parking function of length $ n + 1 + r $, consisting of $r+1$ ones and the integers $2,3,\ldots,n+1$.

We claim $a$ is a $\mathbf{1}_r$-insertion flattened parking function with $k+1$ runs. 
To prove this we begin by noting that in fact the entries in $a$ consist of the elements in the set $[n+1]\cup\mathbf{1}_r$. 
Next we check the number of runs. 
Since the first $r$ integers must appear in distinct blocks, and since we have ordered the blocks by their minimum values, this ensures that $i\in\mathcal{B}_i$ for each $i\in[r]$. 
If all of these sets are singletons, then $1\sigma_1\cdots\sigma_r$ is replaced with $r+1$ ones, which forms the beginning of the first run of $a$.
Whenever $\mathcal{B}_i$, for $i\in[r]$, has size at least two we have introduced a descent through the cycling process used in forming the subword $\sigma_i$. 
As $\mathcal{B}_i$ includes $i$ and all other numbers in $\mathcal{B}_i$ are strictly larger than $r$, then subtracting $r-1$ from those entries keep the descent when forming $a$. 
This shows that the number of descents in $a$ continues to be $k$, which is by assumption the number of blocks with at least two elements. 
Thus, $a$ has $k+1$ runs, as claimed. 

Next we show that $a$ is flattened. First note that since we placed a one at the start of $a$, we are ensured that the first run begins with a one. Next, based on the cycling process and the ordering of the blocks based on their minimums being in increasing order, ensures that the runs in $1\sigma_1\sigma_2\cdots\sigma_\ell$ satisfy that their leading values are in weakly increasing order. Replacing the numbers in $[r]$ with all ones, does not change the fact that the leading terms of the runs are still in weakly increasing order. 
Thus, we know $a$ is in $\textnormal{flat}_{k+1}(\mathcal{PF}_{n+1}(\mathbf{1}_{r}))$.

So far, we have shown that an arbitrary set partition with exactly $k$ blocks of size at least two and whose first $r$ integers appear in distinct blocks, produces an element in $\textnormal{flat}_{k+1}(\mathcal{PF}_{n+1}(\mathbf{1}_{r}))$.
Now, suppose that we have an arbitrary element $\alpha$ in the set $\textnormal{flat}_{k+1}(\mathcal{PF}_{n+1}(\mathbf{1}_{r}))$. 
We construct the set partition of $[n+r]$ where the first $r$ integers appear in distinct blocks and there are exactly $k$ blocks of size at least two, which yields $\alpha$ under our construction above. 
To begin we take the numbers $2,3,\ldots,n+1$ in $a$ and increase them by $r-1$ so that they are now between $r+1, r+2,\ldots, r+n$. 
Note that this keeps all of the descents and runs of $a$.

Now take the $r+1$ ones in $a$ and replace them with the numbers $1,1,2,\ldots, r$ from left to right. 
Note that since those integers were all ones in $a$, by replacing them from left to right with the integers $1,1,2,\ldots, r$, we ensure that these values either remain the leading term of a run (and they are listed in weakly increasing order) or are still within the same run.
Moreover, note that this process yields a word of length $n+1+r$ with the letters in the set $[n+r]\cup\{1\}$,  with $k+1$ runs. 
In Proposition~\ref{prop:k+1Runs_to_kBlocks} we have shown that this is in bijection with a set partition of $[n+r]$ satisfying that there are exactly $k$ blocks with at least two elements. 
It now suffices to show that the first $r$ integers appear in distinct blocks. 
This follows from the bijective map $f^{-1}$ given by Nabawanda et al.\ \cite[Proposition 16]{NaRaBa}, which ensures that our construction satisfies that the numbers $1,2,\ldots, r$ appear in this order from left to right corresponding to the ``left-to-right''
minima, which ensured that each of these elements appears in a distinct block. 
\end{proof}

By summing over $k$, Theorem \ref{rkBellNumbers} immediately implies the following.
\begin{corollary}
    The set $ \textnormal{flat} (\mathcal{PF}_{n+1}(\mathbf{1}_r)) $ is enumerated by $ B(n,r) $, that is the $n^{th}$ $r$-Bell number.
\end{corollary}

We conclude this section by noting that the numbers $B_k(n,r)$ are a generalization of \seqnum{A124234} that include the notion that the first~$r$ integers are in distinct blocks. This sequence of numbers does not yet appear in the OEIS so we provide the Table \ref{tab:my_label} containing the values of $B_k(n,r)$ for small values of $k,n,r$ with $r$ fixed.
\begin{center}
\begin{table}
    \centering
    \begin{tabular}{|c|c|c|c|c|c|}
    \hline         
$n\setminus k$& $1$& $2$&$3$&$4$&$5$\\\hline
         $1$&  1 &0&0&0&0\\
         $2$&  1 &2&0&0&0\\
         $3$&  1 &7&2&0&0\\
         $4$&  1 &18&18&0&0\\
         $5$&  1 &41&97&12&0\\\hline
         \multicolumn{6}{c}{Table of values for $B_k(n,2)$}\\[5pt]\hline
                  $n\setminus k$& $1$& $2$&$3$&$4$&$5$\\\hline
         $1$&  1 &0&0&0&0\\
         $2$&  1 &4&0&0&0\\
         $3$&  1 &13&12&0&0\\
         $4$&  1 &32&84&24&0\\
         $5$&  1 &71&391&312&24\\\hline
         \multicolumn{6}{c}{Table of values for $B_k(n,4)$}\\
         \end{tabular}
         \hspace{.4in}
         \begin{tabular}{|c|c|c|c|c|c|}
         \hline         
$n\setminus k$& $1$& $2$&$3$&$4$&$5$\\\hline
         $1$&  1 &0&0&0&0\\
         $2$&  1 &3&0&0&0\\
         $3$&  1 &10&6&0&0\\
         $4$&  1 &25&45&6&0\\
         $5$&   1&56&219&96&0\\\hline
         \multicolumn{6}{c}{Table of values for $B_k(n,3)$}\\[5pt]\hline
         $n\setminus k$& $1$& $2$&$3$&$4$&$5$\\\hline
         $1$&  1 &0&0&0&0\\
         $2$&  1 &5&0&0&0\\
         $3$&  1 &16&20&0&0\\
         $4$&  1 &39&135&60&0\\
         $5$&  1 &86&613&720&120\\\hline
         \multicolumn{6}{c}{Table of values for $B_k(n,5)$}\\
         \end{tabular}
    \caption{The $r$-Bell numbers, aggregated further by requiring that there are exactly $k$ blocks with at least two elements.}
    \label{tab:my_label}
\end{table}
\end{center}

\section{Generating functions}\label{sec:GenFunc}
In this section, we attempted give a closed differential form for the following exponential generating function:
\begin{equation*}
F(x,y,z) = \sum_{n\geq 1}\sum_{k\geq 1} \sum_{r\geq 1} f(\mathbf{1}_r;n,k)x^{k}\dfrac{y^n}{n!}\frac{z^r}{r!} = \sum_{n\geq 1} \sum_{r\geq 1} \sum_{\alpha \in 
\textnormal{flat} (\mathcal{PF}_n\left(\mathbf{1}_r\right))} x^{\textnormal{run} (\alpha)} \dfrac{y^n}{n!}\frac{z^r}{r!}.
\end{equation*}

Note that, for $r=1$, Theorem \ref{thm:Ones_gen_func_recurs} implies that $ \textnormal{flat} _k(\mathcal{PF}_{n-1}\left(\{1\}\right)) $ and $\mathcal{F}_{n,k}$ are equinumerous, immediately implying that the result of Nabawanda et al.\
\cite[Thm.\ 13, p.\ 8]{NaRaBa} would hold for this case. However, as we attempted to generalize further, we identified an error in the result below. 

\begin{theorem}\label{thm:old_gen_fun}
The exponential generating function $ F(x,y) $ of the run distribution over $ \textnormal{flat} (\mathcal{PF}_{n-1}(\{ 1\})) $ has the closed differential form
\begin{equation*}
\frac{\partial F(x,y)}{\partial y} = x \;\exp(x(\exp(y)-1)+y(1-x))
\end{equation*}
with initial condition $ \frac{\partial F(x,0)}{\partial y} = x $.
\end{theorem}
We now detail our computations which show that Theorem \ref{thm:old_gen_fun} does not hold. 
\begin{remark}\label{rmk:fixed_proof}
    
Nabawanda et al.\ \cite[Thm.\ 13, p.\ 8]{NaRaBa} claimed that
\[
\frac{\partial F}{\partial y} = \sum_{n\geq 1} \sum_{k\geq 1} f_{n+1,k} x^k \frac{y^n}{n!},
\]
which is not quite true. The partial derivative gives us
\[
\frac{\partial F}{\partial y} = \sum_{n\geq 1} \sum_{k\geq 1} f_{n,k} x^k \frac{y^{n-1}}{(n-1)!},
\]
and 
\[
f_{1,k} = \begin{cases}
1, & \text{ when } k=1;\\
0, & \text{ else. }
\end{cases}
\]
Thus,
\[
\frac{\partial F}{\partial y} = x+\sum_{n\geq 1} \sum_{k\geq 1} f_{n+1,k} x^k \frac{y^{n}}{n!},
\]
which, to be fair, does line up with their claim that $ \frac{\partial F(x,0)}{\partial y} = x $. 
Similarly, 
Nabawanda et al.\ \cite[Thm.\ 13, p.\ 8]{NaRaBa} claimed that
\[
\frac{\partial^2 F}{\partial y^2} = \sum_{n\geq 1} \sum_{k\geq 1} f_{n+2,k} x^k \frac{y^n}{n!},
\]
which is also not true. The second partial derivative gives us
\[
\frac{\partial^2 F}{\partial y^2} = \sum_{n\geq 1} \sum_{k\geq 1} f_{n+1,k} x^k \frac{y^{n-1}}{(n-1)!},
\]
and 
\[
f_{2,k} = \begin{cases}
1, & \text{ when } k=1;\\
0, & \text{ else. }
\end{cases}
\]
Thus,
\[
\frac{\partial^2 F}{\partial y^2}= x+\sum_{n\geq 1} \sum_{k\geq 1} f_{n+2,k} x^k \frac{y^{n}}{n!}.
\]
Using this starting point, the result of Nabawanda et al.\ \cite[Thm.\ 12, p.\ 7]{NaRaBa} produces the equality
\[
\frac{\partial^2 F}{\partial y^2} = \frac{\partial F}{\partial y} + \frac{\partial F}{\partial y} \cdot x \cdot (\exp(y) -1).
\]
From our work above, this equality should actually be
\[
\frac{\partial^2 F}{\partial y^2} = x+\left (\frac{\partial F}{\partial y} -x\right )+ \left (\frac{\partial F}{\partial y}-x \right ) \cdot x \cdot (\exp(y) -1).
\]
This reduces to 
\[
\frac{\partial^2 F}{\partial y^2} = \frac{\partial F}{\partial y} + \frac{\partial F}{\partial y}\cdot x \cdot (\exp(y) -1)-x^2 \cdot (\exp(y) -1),
\]
which is not solvable in the way claimed by Nabawanda et al.\ \cite[Proof of Theorem 13]{NaRaBa}.
\end{remark}

Given the content of our remark above and the fact that our work relies on theirs, we conclude by providing the details of our analysis and where the lack of a solution to their differential equation prevents us from arriving at a closed form solution. 

\begin{remark}\label{Fxyzgenfunction}
We have
\begin{equation*}
F(x,y,z) = \sum_{n\geq1} \sum_{k\geq1} \sum_{r\geq 1} f(\mathbf{1}_r;n,k) x^{k} \cdot \frac{z^r}{r!}\frac{y^{n}}{n!} = \sum_{n\geq1} g_{n}(x,z) \frac{y^{n}}{n!},
\end{equation*}
where $ g_{n}(x,z) $ is the polynomial defined by $ \sum_{k\geq1}\sum_{r\geq 1} f(\mathbf{1}_r;n,k) x^{k} \frac{z^r}{r!} $.

First, we see that
\begin{align*}
\int_{0}^{z} F(x,y,t) dt &= \int_{0}^{z} \sum_{n\geq1} \sum_{k\geq1} \sum_{r\geq 1} f(\mathbf{1}_r;n,k) x^{k} \cdot \frac{t^r}{r!}\frac{y^{n}}{n!} dt\\
&= \sum_{n\geq1} \sum_{k\geq1} \sum_{r\geq 1} f(\mathbf{1}_r;n,k) x^{k} \cdot \frac{z^{r+1}}{(r+1)!}\frac{y^{n}}{n!}.
\end{align*}

Note that 
\[
\frac{\partial F}{\partial y} = \sum_{n\geq 1}\sum_{k\geq 1} \sum_{r\geq 1} f(\mathbf{1}_r;n,k)x^{k}\dfrac{y^{n-1}}{(n-1)!}\frac{z^r}{r!},
\]
where
\[
f(\mathbf{1}_r;1,k)=\begin{cases} 1, & \text{when $k=1$;}\\
0, & \text{else.}
\end{cases}
\]
This means that, when $n=1$, we have
\[
\sum_{k\geq 1}\sum_{r\geq 1} f(\mathbf{1}_r;1,k)x^{k}\frac{z^r}{r!} = \sum_{r\geq 1} f(\mathbf{1}_r;1,1)x^{1}\frac{z^r}{r!}= x\cdot \exp(z),
\]
and thus
\begin{align}
\frac{\partial F}{\partial y} &= x\cdot \exp(z) +  \sum_{n\geq 2}\sum_{k\geq 1} \sum_{r\geq 1} f(\mathbf{1}_r;n,k)x^{k}\dfrac{y^{n-1}}{(n-1)!}\frac{z^r}{r!}\\
&= x\cdot \exp(z) +  \sum_{n\geq 1}\sum_{k\geq 1} \sum_{r\geq 1} f(\mathbf{1}_r;n+1,k)x^{k}\dfrac{y^{n}}{n!}\frac{z^r}{r!}.
\end{align}

Shortly, we use this in order to see that
\[
\frac{\partial F}{\partial y} - x\cdot \exp(z) = \sum_{n\geq 1}\sum_{k\geq 1} \sum_{r\geq 1} f(\mathbf{1}_r;n+1,k)x^{k}\dfrac{y^{n}}{n!}\frac{z^r}{r!}.
\]

From Theorem \ref{thm:R_Ones_gen_func_recurs}, and summing over $ n $, and $ k $, and $r$, we have
\begin{align*}
\frac{\partial F}{\partial y} &=  x\cdot \exp(z) + \sum_{n\geq1} \sum_{k\geq1} \sum_{r\geq 1} f(\mathbf{1}_r;n+1,k)x^{k} \cdot\frac{z^r}{r!} \frac{y^{n}}{n!}\\
& = x\cdot \exp(z)+\sum_{n\geq1} \sum_{k\geq1} \sum_{r\geq 1} f(\mathbf{1}_{r-1};n+1,k) x^{k} \cdot \frac{z^r}{r!}\frac{y^{n}}{n!}\\
& \quad + \sum_{n\geq1} \sum_{k\geq1} \sum_{r\geq 1} \left(\sum_{i=1}^{n} \binom{n}{ i} f(\mathbf{1}_{r-1};(n+1)-i,k-1) \right) x^{k} \cdot \frac{z^r}{r!}\frac{y^{n}}{n!}.
\end{align*}

We define $B$ as follows:
\[
B\coloneqq \sum_{n\geq1} \sum_{k\geq1} \sum_{r\geq 1} f(\mathbf{1}_{r-1};n+1,k) x^{k} \cdot \frac{z^r}{r!}\frac{y^{n}}{n!}.
\]

We note that simplifying B gives us (using both the derivative and integral forms we established above),
\begin{align*}
B &= \sum_{n\geq1} \sum_{k\geq1}  f(\mathbf{1}_{0};n+1,k)  x^{k} \cdot \frac{z^1}{1!}\frac{y^{n}}{n!} + \sum_{n\geq1} \sum_{k\geq1} \sum_{r\geq 1} f(\mathbf{1}_{r};n+1,k)  x^{k} \cdot \frac{z^{r+1}}{(r+1)!}\frac{y^{n}}{n!}\\
&= z\left ( \sum_{n\geq1} \sum_{k\geq1} f_{n+1, k} x^k\frac{y^n}{n!}\right ) + \int_0^z \left(\frac{\partial F(x,y,t)}{\partial y} - x\exp(t) \right ) dt.
\end{align*}

According to Remark \ref{rmk:fixed_proof}, the result of Nabawanda et al.\
\cite[Thm.\ 13, p.\ 8]{NaRaBa} should be
\[
\sum_{n\geq1} \sum_{k\geq1} f_{n+1, k} x^k\frac{y^n}{n!} =  \frac{\partial F(x,y)}{\partial y} -x. \]

This implies that
\[
B = z\left ( \frac{\partial F(x,y)}{\partial y} -x \right ) +\int_0^z \left(\frac{\partial F(x,y,t)}{\partial y} - x\exp(t) \right ) dt.
\]

Let 
\[
C=\sum_{n\geq1} \sum_{k\geq1} \sum_{r\geq 1} \left(\sum_{i=1}^{n} \binom{n}{i} f(\mathbf{1}_{r-1};(n+1)-i,k-1) \right) x^{k} \cdot \frac{z^r}{r!}\frac{y^{n}}{n!}.
\]

Fixing $ i $ and summing over $ k $ and $ r $ in the next triple sum gives
\[
C = \sum_{n\geq 1} \sum_{i=1}^{n} \left( \sum_{k\geq1} \sum_{r\geq 1} \binom{n}{i} f(\mathbf{1}_{r-1};(n+1)-i,k-1) x^{k} \cdot \frac{z^r}{r!}\right) \frac{y^{n}}{n!}.
\]

Expanding the binomial coefficient for $ C $, and factoring out an $ x $, we have
\begin{align*}
C & = x \sum_{n\geq1} \sum_{i=1}^{n} \frac{y^i}{i!}\left( \sum_{k\geq1} \sum_{r\geq 1}f(\mathbf{1}_{r-1};n-i+1,k-1) x^{k-1} \cdot\frac{y^{n-i}}{(n-i)!} \frac{z^r}{r!}\right).
\end{align*}

Because $n\geq 1$ and $i$ ranges from one to $n$, we know $n-i$ ranges from zero to $n-1$. Simplifying and swapping the order of our sums and reindexing the inner sum indexed by $n$ gives
\begin{align*}
C & = x \sum_{i\geq1} \frac{y^i}{i!} \left ( \sum_{n\geq i}  \sum_{k\geq1} \sum_{r\geq 1} f(\mathbf{1}_{r-1};n-i+1,k-1) x^{k-1} \cdot\frac{y^{n-i}}{(n-i)!} \frac{z^r}{r!}\right) .
\end{align*}

We have
\begin{equation*}
C = x \cdot B \cdot \left(\exp(y) - 1\right).
\end{equation*}

Substituting and simplifying, we get
\[
\frac{\partial F}{\partial y} = x\cdot \exp(z)+ B+C.
\]
Since we also were unable to give a closed form solution for $B$, we stopped our analysis here. 
Thus, we conclude with this open problem:
Consider the exponential generating function \[F(x,y,z) = \sum_{n\geq 1}\sum_{k\geq 1} \sum_{r\geq 1} f(\mathbf{1}_r;n,k)x^{k}\dfrac{y^n}{n!}\frac{z^r}{r!} = \sum_{n\geq 1} \sum_{r\geq 1} \sum_{\alpha \in 
\textnormal{flat} (\mathcal{PF}_n\left(\mathbf{1}_r\right))} x^{\textnormal{run} (\alpha)} \dfrac{y^n}{n!}\frac{z^r}{r!},\] and give a closed form. 
\end{remark}

\section{ \texorpdfstring{$\mathbf{1}_{r}$}{Lg}-insertion flattened parking functions with the first \texorpdfstring{$s$}{Lg} terms in different runs}\label{sec:flattened_s_in_dif_runs}

In this section, we consider $r\geq 1$.
We begin by stating the recursive result of Nabawanda et al.\ \cite[Thm.\ 14, p.\ 10]{NaRaBa}, 
which considers the first $s$ integers in $[n]$ being in separate runs.

\begin{theorem}\label{thm:Olivia_rec}
For all integers $s$, $k$, and $n$ satisfying $1 \leq s \leq  k < n$, let $ f^{(s)}(n+s,k) $ represent the number of flattened partitions where the first $s$ distinct integers are in separate runs. The numbers $f^{(s)}(n+s,k)$ satisfy the recursion 
\begin{equation*}
f^{(s+1)}(n+s+1,k)= \sum_{i_1,i_2,\ldots,i_s\geq1} \binom{n}{i_1,\ldots,i_{s}} f(n+1-\sum_{j=1}^{s} i_{j}\,,\,k-s),
\end{equation*}
where we let $f(n,k) \coloneqq f_{n,k}$ denote the number of flattened partitions of length $n$ with $k$ runs.
\end{theorem}

Let $ f^{(s)} (\mathbf{1}_r; n+s, k)$ represent the number of  $\mathbf{1}_{r}$-insertion flattened parking functions where the first $s$ integers are in separate runs. Note that this means that $k\geq s$.

There are many cases to consider here: the ones can be separate from all the other $s-1$ integers in many ways. Since there are $r+1$ instances of the integer one in all $\mathbf{1}_r$-insertion flattened parking functions, we would have to consider all compositions of the integer $r+1$ as options for how to distribute the ones among the $k$ runs. 
Our general recurrence for $ f^{(s)} (\mathbf{1}_r; n+s, k)$ takes into account all such options. 
However, we begin with some special cases that are used in the proof of our main result which are also direct generalizations of the work of Nabawanda et al.\ \cite{NaRaBa}.

In the following result we suppose that all of the ones are contained in the same run. In this case, we recover the result of 
Nabawanda et al.\ \cite[Thm.\ 14, p.\ 10]{NaRaBa}.

\begin{theorem}\label{thm:s_split_v1}
For all integers $s$, $k$, $n$, and $r$ satisfying $1 \leq s < k < n$ and $r\geq1$, let $g^{(s)}(\mathbf{1}_r;n+s,k)$ denote the number of  $\mathbf{1}_{r}$-insertion flattened parking functions where the first $s$ distinct integers are in separate runs and all ones are in the same run. The numbers $g^{(s)}(\mathbf{1}_r;n+s,k)$ satisfy the recursion 
\begin{equation*}
g^{(s+1)}(\mathbf{1}_r;n+s+1,k)= \sum_{i_1,i_2,\ldots,i_s\geq1} \binom{n}{i_1,\ldots,i_{s}} f(n+1-\sum_{j=1}^{s} i_{j},k-s).
\end{equation*}
\end{theorem}

\begin{proof}
Let $\pi$ be an element in $\textnormal{flat} _k(\mathcal{PF}_{n+s+1}\left(\mathbf{1}_r\right))
$ where the first $s+1$ distinct integers are in separate runs and all ones are in the first run. To maintain flattenedness, the first $s+1$ runs begin with $1,2, \ldots, s, s+1$, respectively. 
Let $i_1+r+1, i_2+1, i_3+1, \ldots, i_s +1$ be the lengths of the first $s$ runs. 
Note that the leading terms of each run, as well as the placement of the ones, are fixed. 
Thus, we have $i_1$ terms to arrange out of $(n + s + 1 + r) - (s + 1 + r) = n$ 
terms for the first run, and there are $\binom{n}{i_1}$ ways to do so. 
For the second run, we have $i_2$ terms to arrange out of $n-i_1$ terms, and there are $\binom{n-i_1}{i_2}$  ways to do so. 
Repeating this process $s$ times, yields $\binom{n}{i_{1}} \binom{n-i_1}{i_{2}} \cdots \binom{n-\sum_{j=1}^{s-1} i_{j}}{i_{s}} = \binom{n}{i_1,\ldots,i_{s}}$  possibilities for the first way to arrange the entries in the first $s$ runs.
Now that the first $s$ runs have been chosen, we consider the remaining $n+1-\sum_{j=1}^{s} i_{j}$ elements that are arranged into the final $k-s$ runs. The smallest number in this set is $s+1$ by design, so it starts the $(s+1)$th run. 
There are $f(n+1-\sum_{j=1}^{s} i_{j},k-s)$ possible ways to arrange the remaining elements into a flattened subword. Thus, we have a total of $\sum_{i_1,i_2,\ldots,i_s\geq1} \binom{n}{i_1,\ldots,i_{s+r}} f(n+1-\sum_{j=1}^{s} i_{j},k-s)$ words that are $\mathbf{1}_r$-insertion flattened parking functions where the first $s+1$ distinct integers are in separate runs and all of the ones are in the first run, as desired.
\end{proof}

Next, we consider the case where each one is contained in a distinct run, which again recovers the result of Nabawanda et al.\ \cite[Thm.\ 14, p.\ 10]{NaRaBa}.

\begin{theorem}\label{thm:s_split_v2}
For all integers $s$, $k$, $n$, and $r$ satisfying $ 1 \leq s + r < k < n $, let $ h^{(s)}(\mathbf{1}_r;n+s,k)$ denote the number of $ \mathbf{1}_{r} $-insertion flattened parking functions where the $r$ ones and first $s$ distinct integers are all in separate runs. The numbers $h^{(s)}(\mathbf{1}_r;n+s,k)
$ satisfy the recursion 
\begin{equation*}
h^{(s+1)}(\mathbf{1}_r;n+s+1,k)
 = \sum_{i_{1},\ldots,i_{s+r}\geq1} \binom{n}{i_{1},\ldots,i_{s+r}} f(n+1-\sum_{j=1}^{s+r} i_{j},k-s-r).
\end{equation*}
\end{theorem}

\begin{proof}
Let $\pi$ be an element in $\textnormal{flat} _k(\mathcal{PF}_{n+s}\left(\mathbf{1}_r\right)
)$ where the $r$ ones and first $s+1$ distinct integers are in separate runs. 
Let $i_1+1, i_2+1, i_3+1, \ldots, i_{s+r} +1$ be the lengths of the first $s+r$ runs. Note that, to preserve flattenedness, the leading terms of each run are fixed. Thus, we have $i_1$ terms to arrange out of $(n + s + r + 1) - (s + r + 1) = n$ terms for the first run, and there are $\binom{n}{i_1}$ ways to do so. 
For the second run, we have $i_2$ terms to arrange out of $n-i_1$ terms, and there are $\binom{n-i_1}{i_2}$ ways to do so. Repeating this process $s+r$ times, we have 
$\binom{n }{ {i_1}} 
\binom{{n-i_1} }{ {i_2}} 
\cdots 
\binom{n-\sum_{j=1}^{s-1} i_j}{ {i_s}} = \binom{n}{{i_1,\ldots,i_{s+r}}}$ possibilities for the first $s+r$ runs. The remaining $k-s-r$ runs are determined by the integers not used in the first $s+r$ runs. The smallest number in this set is $s+1$ by design, so it starts the $(s+1+r)$th run. There are $f(n+1-\sum_{j=1}^{s+r} i_{j}, k-s-r)$ possibilities for the pattern of these integers. Thus, we have a total of $\sum_{i_1,i_2,\ldots,i_{s+r}\geq1} \binom{n}{{i_1,\ldots,i_{s+r}}} f(n+1-\sum_{j=1}^{s+r} i_{j},k-s-r)$ words that are $\mathbf{1}_r$-insertion flattened parking functions where the first $s+1$ distinct integers and the additional $r$ ones are in separate runs.
\end{proof}

\begin{theorem}\label{thm:s_split_all_ways}
For all integers $ s,k,n, $ and $ r $ such that $ 1 \leq s+r < k < n $, let $ f^{(s)}(\mathbf{1}_r;n+s,k)$ denote the number of $ \mathbf{1}_{r} $-insertion flattened parking functions where the first $ s $ distinct integers are all in separate runs, with the ones in every possible composition of runs. The numbers $f^{(s)}(\mathbf{1}_r;n+s,k)$ satisfy the recursion 
{{
\begin{align*}
f^{(s+1)}(\mathbf{1}_r;n+s+1,k)
 = \hspace{4.5in}\\
 \sum_{x=1}^{r+1} \left(\binom{r }{ x-1} \cdot\left( \sum_{i_{1},\ldots,i_{s+x-1}\geq1} \binom{n}{{i_{1},\ldots,i_{s+x-1}}} f\left(n+1-\sum_{j=1}^{s+x-1} i_{j},k-s-x+1\right)\right)\right).
\end{align*}
}}
\end{theorem}

\begin{proof}
Let $\pi$ be an element in $\textnormal{flat} _k(\mathcal{PF}_{n+s}\left(\mathbf{1}_r\right)
)$ where the $r$ ones and first $s+1$ distinct integers are in separate runs. 
We begin by counting the number of $\mathbf{1}_{r}$-insertion flattened parking functions such that the $r+1$ ones are in $x$ different runs, with at least a single one in each of the $x$ runs, note that this necessitates $1 \leq x \leq r+1$. 
Recall from Theorem \ref{thm:s_split_v2} that $ \sum_{i_1,i_2,\ldots,i_{s+x}\geq1} \binom{n}{{i_1,\ldots,i_{s+x}}} f(n+1-\sum_{j=1}^{s+x} i_{j},k-s-x) $ counts the number of $\mathbf{1}_x$-insertion flattened parking functions with the first $s+1$ integers in different runs. 
Observe that there are $r+1-x$ ones that are missing for $\pi$ to be a $\mathbf{1}_{r}$-insertion flattened parking function. 
We can arrange these ones in any of the first $x$ runs without changing all other runs of $\pi$, and this continues to be flattened. 

The number of ways to arrange the ones in the first $ x $ runs is the same as the number of integer compositions of $ r+1 $ over $ x $ parts.
Each part of such a composition indicates the number of ones in each run.
This means that we can multiply by a coefficient of $ \binom{r }{ x-1} $ to get the total number of $ \mathbf{1}_{r} $-insertion flattened parking functions with the first $s+1$ integers in different runs and all ones over $x$ runs. 

Then, to find the total number of $\mathbf{1}_{r}$-insertion flattened parking functions with the first $s+1$ integers in different runs, we include the instances where all ones are in the same run, two different runs, up to $r+1$ different runs. Hence the result is given  by a sum over all possible values for $x-1$, i.e.,~$x=1,2,\ldots,r+1$.
\end{proof}

\begin{remark}\label{gf_fun_2}
As in the previous section, we attempted to give a generalization of the result of Nabawanda et al.\ \cite[Thm.\ 15, p.\ 10]{NaRaBa}. However, we again identified an error given the incorrect computation of Nabawanda et al.\ \cite[Thm.\ 13, p.\ 8]{NaRaBa}. 
Given those errors, we ceased our analysis. Thus, we conclude with this open problem:
Consider the exponential generating function $ F^{[s+1]} (x,u) $ for the numbers $ h^{(s+1)}\left(\mathbf{1}_{r},n+s+1, k\right) $ and give a closed form. 
\end{remark}

\section{Future directions}

Aside from the generating function questions we posed in Remark \ref{Fxyzgenfunction} and Remark \ref{gf_fun_2}, we present two additional directions for further study: 1. the insertion process and bijections, 2. enumerating flattened parking functions with $k$ runs. We expand on these future directions below. 

In this paper, we introduced the $\mathcal{S}$ insertion process for permutations. which always created a parking function. 
We wonder if there are other bijections between related insertion sets, similar to the results we presented in Section \ref{section:general_S_insertion}. One could also expand to other $\mathcal{S}$ consisting of other repeated values to see what of our results can be generalized further.

Flattened parking functions still need more consideration. We provided some initial data on the enumeration of flattened parking functions of length $n$ with $k$ runs, but these sequences do not appear in the OEIS. Thus, it remains an open problem to provide recurrences or formulas enumerating these combinatorial objects.

\section{Acknowledgments} 
The authors were supported by the National Science Foundation under Grant No.\ DMS-1929284 while in residence at the Institute for Computational and Experimental Research in Mathematics in Providence, RI.
We thank ICERM for funding support and extend our gratitude to Susanna Fishel and Gordon Rojas Kirby for many helpful conversations throughout the writing of this manuscript. The authors also extend their thanks to Alexander Woo for key insights which helped establish Lemma~\ref{prop:help_from_A}.
P.~E.~Harris was supported through a Karen Uhlenbeck EDGE Fellowship.


\begin{thebibliography}{99}

\bibitem{DAet}{D.\ Armstrong, N.\ A.\ Loehr, and G.\ S.\ Warrington,
Rational parking functions and Catalan numbers, \textit{Ann. Comb.} \textbf{20} (2016), 21--58.}

\bibitem{FBRM}{
F.\ Beyene and R.\ Mantaci, Merging-free partitions and run-sorted permutations, \textit{J. Integer Seq.} \textbf{25} (2022), Article 22.7.6.}


\bibitem{Harris1}{A.\ Diaz-Lopez, P.\ E.\ Harris, E.\ Insko, M.\ Omar, and B.\ Sagan, Descent polynomials, \textit{Discrete Math.} \textbf{342} (2019), 1674--1686.}

\bibitem{JEetal}{J.\ Elder, N.\ Lafreni\'ere, E.\ McNicholas, J.\ Striker, and A.\ Welch, Homomesies on permutations: an analysis of maps and statistics in
the FindStat database, arxiv preprint arXiv:2206.13409 [math.CO], 2022. Available at \url{https://arxiv.org/abs/2206.13409}.}

\bibitem{IG}{
I.\ M.\ Gessel, 
Generating functions and enumeration of sequences, 
Thesis (Ph.D.)--Massachusetts Institute of Technology, 1977,
\url{http://dspace.mit.edu/handle/1721.1/16342}.}

\bibitem{KW}{A.\ Konheim and B.\ Weiss, An occupancy discipline and applications, \textit{SIAM J. Appl. Math.} \textbf{14} (1966), 1266--1274.}

\bibitem{TMMSSW}{T.\ Mansour, M.\ Shattuck, and S. Wagner, Counting subwords in flattened partitions of sets, \textit{Discrete Math.} \textbf{338} (2015), 1989--2005}.

\bibitem{NaRaBa}{O. Nabawanda, F.\ Rakotondrajao, and A.\ S. Bamunoba, Run distribution over flattened partitions, \textit{J. Integer Seq.} \textbf{23} (2020), Article 20.9.6.}


\bibitem{SCH}{P.\ R.\ F.\ Schumacher, Descents in parking functions, \textit{J. Integer Seq.} \textbf{21} (2018), Article 18.2.3.}


\bibitem{OEIS}{N.\ J.\ A.\ Sloane et al., \textit{The On-Line Encyclopedia of Integer Sequences}, 2022. Available at \url{https://oeis.org}.}


\bibitem{JS}{J.\ Steinhardt, Permutations with ascending and descending blocks, \textit{Electron. J. Combin.} \textbf{17} (2010), \#R14.}

\bibitem{Yan}{
C.\ H.\ Yan, Parking Functions, in
M.\ B\'{o}na, ed., \textit{Handbook of Enumerative Combinatorics}, CRC Press, 2015 pp.\ 835--893.}


\bibitem{YZ}{Y.\ Zhuang, Counting permutations by runs, \textit{J. Combin. Theory Ser. A} \textbf{142} (2016), 147--176}.




\end{thebibliography}
\end{document}